\documentclass[a4paper,11pt,reqno]{amsart}
\usepackage[utf8]{inputenc}
\usepackage{amsmath}
\usepackage{amssymb}
\usepackage{amsthm}
\usepackage{hyperref}
\usepackage{color}
\usepackage{a4wide}
\usepackage[arrow, matrix, curve]{xy}
\makeatletter

\newcommand{\fraka}{\mathfrak{a}}
\newcommand{\frakg}{\mathfrak{g}}

\newcommand{\frakm}{\mathfrak{m}}
\newcommand{\frakn}{\mathfrak{n}}
\newcommand{\frako}{\mathfrak{o}}

\newcommand{\fraku}{\mathfrak{u}}

\newcommand{\CC}{\mathbb{C}}
\newcommand{\DD}{\mathbb{D}}
\newcommand{\FF}{\mathbb{F}}

\newcommand{\NN}{\mathbb{N}}

\newcommand{\RR}{\mathbb{R}}
\newcommand{\ZZ}{\mathbb{Z}}
\newcommand{\calD}{\mathcal{D}}
\newcommand{\calE}{\mathcal{E}}
\newcommand{\calF}{\mathcal{F}}

\newcommand{\calL}{\mathcal{L}}
\newcommand{\calO}{\mathcal{O}}
\newcommand{\calP}{\mathcal{P}}

\newcommand{\calS}{\mathcal{S}}
\newcommand{\calT}{\mathcal{T}}
\newcommand{\calU}{\mathcal{U}}
\newcommand{\calV}{\mathcal{V}}

\newcommand{\0}{{\bf 0}}
\renewcommand{\1}{{\bf 1}}

\DeclareMathOperator{\Ind}{Ind}
\DeclareMathOperator{\tr}{tr}
\DeclareMathOperator{\ptr}{ptr}
\DeclareMathOperator{\ad}{ad}
\DeclareMathOperator{\Ad}{Ad}

\DeclareMathOperator{\HS}{HS}

\DeclareMathOperator{\id}{id}
\DeclareMathOperator{\sgn}{sgn}

\DeclareMathOperator{\diag}{diag}
\DeclareMathOperator{\blank}{--}
\DeclareMathOperator{\vol}{vol}
\DeclareMathOperator{\singsupp}{sing\,supp}
\renewcommand\Re{\operatorname{Re}}
\renewcommand\Im{\operatorname{Im}}

\theoremstyle{plain}
\newtheorem{theorem}{Theorem}[section]
\newtheorem{proposition}[theorem]{Proposition}
\newtheorem{lemma}[theorem]{Lemma}
\newtheorem{corollary}[theorem]{Corollary}

\newtheorem{fact}[theorem]{Fact}

\newtheorem{thmalph}{Theorem}
\newtheorem{conjalph}[thmalph]{Conjecture}

\theoremstyle{definition}

\newtheorem{remark}[theorem]{Remark}

\numberwithin{equation}{section}

\title[On boundary value problems for some differential operators]{On boundary value problems for some conformally invariant differential operators}
\author{Jan M\"{o}llers}
\author{Bent {\O}rsted}
\author{Genkai Zhang}

\address{Department of Mathematics, The Ohio State University, 231 West 18th Avenue, Columbus, OH 43210, USA}
\email{mollers.1@osu.edu}

\address{Institut for Matematiske Fag, Aarhus Universitet, Ny Munkegade 118, 8000 Aarhus C, Denmark}
\email{orsted@imf.au.dk}

\address{Mathematical Sciences, Chalmers University of Technology and the University of Gothenburg, SE-412 96 G\"{o}teborg, Sweden}
\email{genkai@chalmers.se}
\thanks{Research by G. Zhang partially supported by the Swedish Science Council (VR)}

\begin{document}

\begin{abstract}
We study boundary value problems for some differential operators on
Euclidean space and the Heisenberg group which are invariant under the
conformal group of a Euclidean subspace resp. Heisenberg subgroup. These operators are shown to be self-adjoint in certain Sobolev type spaces and the related boundary value problems are proven to have unique solutions in these spaces. We further find the corresponding Poisson transforms explicitly in terms of their integral kernels and show that they are isometric between Sobolev spaces and extend to bounded operators between certain $L^p$-spaces.\\
The conformal invariance of the differential operators allows us to apply unitary representation theory of reductive Lie groups, in particular recently developed methods for restriction problems.
\end{abstract}

\maketitle

\section{Introduction}

Some of the most important elliptic boundary value problems are geometric in nature; one aspect of this is that there is a Lie group of symmetries acting on the space of solutions. A classical example is provided by the harmonic functions in the complex upper half-plane, which one wants to study via their boundary values on the real axis. The classical Poisson transform
$$ Pf(x,y) = \frac{1}{\pi}\int_{-\infty}^\infty\frac{y}{(x-x')^2+y^2}f(x')dx' $$
relates the boundary value and the solution, and it is well known that it is invariant for the projective group (the M\"{o}bius group) of the real axis. It turns out that some recently studied boundary value problems (see Caffarelli--Silvestre~\cite{CS07}) exhibit similar, and perhaps a little overlooked, symmetries. Thus in this paper we apply facts from the theory of unitary representations of higher-dimensional M\"{o}bius groups to exhibit
\begin{enumerate}
\item new and natural Sobolev spaces of solutions and unitary Poisson transforms acting between boundary values and solutions, and
\item new Poisson transforms related to branching problems for the unitary representations, namely restricting from one group to a subgroup.
\end{enumerate}

Abstractly speaking our study relates solutions on a flag manifold to their restrictions (boundary values) to a natural flag submanifold; as such it is natural to extend to other geometries, and we carry this out for the case of CR geometry, another case of recent interest (see Frank et al.~\cite{FGMT15}). It seems to be a promising outlook to extend to other cases of semi-simple Lie groups and subgroups in analogy with those considered here, to relate representation theory and elliptic boundary value problems. Branching theory for unitary representations of semi-simple Lie groups contains structures that might well cast new light on more general boundary value problems. In particular natural candidates for Poisson transforms are provided by the class of
symmetry-breaking operators constructed in the recent work of T.~Kobayashi and B.~Speh~\cite{KS13}, and also J.~M\"{o}llers, Y.~Oshima, and B.~{\O}rsted~\cite{MOO13}.

We now state our results in detail.

\subsection{Euclidean space}

For $a\in\RR$ we consider the differential operator
$$ \Delta_a = x_n^2\Delta+ax_n\frac{\partial}{\partial x_n} $$
on $\RR^n$ where $\Delta$ denotes the Laplacian. In \cite{CS07} Caffarelli--Silvestre study the Dirichlet problem
$$ \Delta_au = 0, \qquad u|_{\RR^{n-1}} = f. $$
The following result provides a suitable Hilbert space setting for this problem:

\begin{thmalph}\label{thm:DirichletReal}
\begin{enumerate}
\item For $2-n<a\leq2$ the operator $\Delta_a$ is essentially self-adjoint on the homogeneous Sobolev space $\dot{H}^{\frac{2-a}{2}}(\RR^n)$. Its spectrum $\sigma(\Delta_a)=\sigma_p(\Delta_a)\cup\sigma_c(\Delta_a)$ is given by
$$ \sigma_p(\Delta_a) = \{k(k+a-1):k\in\NN,\,k<\tfrac{1-a}{2}\}, \qquad \sigma_c(\Delta_a) = \left(-\infty,-(\tfrac{1-a}{2})^2\right). $$
\item For $2-n<a<1$ and $f\in \dot{H}^{\frac{1-a}{2}}(\RR^{n-1})$ the Dirichlet problem
\begin{equation}
 \Delta_au=0, \qquad u|_{\RR^{n-1}}=f\label{eq:RealBdyValueProblem}
\end{equation}
has a unique solution $u\in\dot{H}^{\frac{2-a}{2}}(\RR^n)$. This solution is even in the variable $x_n$.
\end{enumerate}
\end{thmalph}

Caffarelli--Silvestre further investigate the Poisson transform $P_a$ for the Dirichlet problem, mapping the boundary values $f$ to the solutions $u$, and find its integral kernel:
\begin{equation}
 P_af(x) = c_{n,a} \int_{\RR^{n-1}} \frac{|x_n|^{1-a}}{(|x'-y|^2+x_n^2)^{\frac{n-a}{2}}} f(y) \,dy\label{eq:PoissonTrafoReal}
\end{equation}
with $c_{n,a}=\pi^{-\frac{n-1}{2}}\Gamma(\frac{n-a}{2})\Gamma(\frac{1-a}{2})^{-1}$. We show the following properties of $P_a$:

\begin{thmalph}\label{thm:IsometryReal}
Assume $2-n<a<1$.
\begin{enumerate}
\item The operator $P_a:\dot{H}^{\frac{1-a}{2}}(\RR^{n-1})\to\dot{H}^{\frac{2-a}{2}}(\RR^n)$ is isometric up to a constant. More precisely,
\begin{equation}
 \|P_af\|_{\dot{H}^{\frac{2-a}{2}}(\RR^n)}^2 = \frac{2^a\pi\Gamma(2-a)}{\Gamma(\frac{1-a}{2})\Gamma(\frac{3-a}{2})}\|f\|_{\dot{H}^{\frac{1-a}{2}}(\RR^{n-1})}^2, \qquad f\in\dot{H}^{\frac{1-a}{2}}(\RR^{n-1}).\label{eq:IsometryReal1}
\end{equation}
\item The Poisson transform $P_a$ extends to a bounded operator $P_a:L^p(\RR^{n-1})\to L^q(\RR^n)$ for any $1<p\leq\infty$ and $q=\frac{n}{n-1}p$. More precisely,
$$ \|P_af\|_{L^q(\RR^n)} \leq (2c_{n,a}^{\frac{1}{n-1}})^{\frac{1}{q}}\|f\|_{L^p(\RR^{n-1})}, \qquad f\in L^p(\RR^{n-1}). $$
\end{enumerate}
\end{thmalph}

Part (2) of Theorem~\ref{thm:IsometryReal} has been proven earlier by Chen~\cite{Che14}. He even shows that for the particular parameters $p=\frac{2(n-1)}{n-2+a}$ and $q=\frac{2n}{n-2+a}$ there exists a sharp constant $C>0$ such that
$$ \|P_af\|_{L^q(\RR^n)} \leq C\|f\|_{L^p(\RR^{n-1})}, \qquad f\in L^p(\RR^{n-1}), $$
and that the optimizers of this inequality are translations, dilations and multiples of the function
$$ f(y) = (1+|y|^2)^{-\frac{a+n-2}{2}}. $$
By the classical Hardy--Littlewood--Sobolev inequality we have
$\dot{H}^{\frac{1-a}{2}}(\RR^{n-1})\subseteq L^p(\RR^{n-1})$ and
$\dot{H}^{\frac{2-a}{2}}(\RR^n)\subseteq L^q(\RR^n)$ and hence our
isometry property in Theorem~\ref{thm:IsometryReal}~(1) can be viewed
as a Hilbert space version of Chen's sharp inequality; namely
we have the following commutative diagram with the respective
boundedness properties:
$$
\begin{xy}
\xymatrix{
 \dot{H}^{\frac{1-a}{2}}(\RR^{n-1}) \ar@{^{(}->}[d] \ar[r]^{P_a} & \dot{H}^{\frac{2-a}{2}}(\RR^n) \ar@{^{(}->}[d]\\
 L^p(\RR^{n-1}) \ar[r]^{P_a} & L^q(\RR^n).
}
\end{xy}
$$

While the Dirichlet problem corresponds to the eigenvalue $0$ of $\Delta_a$, one can more generally associate to every eigenvalue $k(k+a-1)$, $0\leq k<\frac{1-a}{2}$, a mixed boundary value problem
$$ \Delta_au = k(k+a-1)u, \qquad D_{a,k}u=f, $$
where
$$ D_{a,k}u = \left.\Bigg(p\Big(\sum_{j=1}^{n-1}\frac{\partial^2}{\partial x_j^2},\frac{\partial}{\partial x_n}\Big)u\Bigg)\right|_{\RR^{n-1}} $$
is a differential operator of order $k$ with $p(x,y)$ a certain polynomial on $\RR^2$ which can be defined in terms of the classical Gegenbauer polynomials. The operators $D_{a,k}$ were found by Juhl~\cite{Juh09} and the corresponding Poisson transforms $P_{a,k}$ are given in Corollary~\ref{cor:MixedBVPReal} (see Appendix~\ref{sec:AppendixReal} for the full spectral decomposition of $\Delta_a$). We remark that $D_{a,1}u=\frac{\partial u}{\partial x_n}|_{\RR^{n-1}}$ and hence the corresponding mixed boundary value problem for $k=1$ is the Neumann problem
$$ \Delta_au=au, \qquad \left.\frac{\partial u}{\partial x_n}\right|_{\RR^{n-1}}=g. $$

\subsection{The Heisenberg group}

Now consider the Heisenberg group $H^{2n+1}=\CC^n\oplus\RR$ with multiplication given by
$$ (z,t)\cdot(z',t') = (z+z',t+t'+2\Im(z\cdot\overline{z'})). $$
Further, let
$$ |(z,t)| = (|z|^4+t^2)^{\frac{1}{4}} $$
denote the norm function on $H^{2n+1}$. We study the following differential operator on $H^{2n+1}$:
$$ \calL_a = |z_n|^2\calL+a\left(x_n\frac{\partial}{\partial x_n}+y_n\frac{\partial}{\partial y_n}\right), $$
where $\calL$ is the CR Laplacian on $H^{2n+1}$ given by
$$ \calL = \sum_{j=1}^n\left(\left(\frac{\partial}{\partial x_j}+2y_j\frac{\partial}{\partial t}\right)^2+\left(\frac{\partial}{\partial y_j}-2x_j\frac{\partial}{\partial t}\right)^2\right). $$
Using the group Fourier transform on $H^{2n+1}$ one can define natural Sobolev type spaces $\dot{H}^{s}(H^{2n+1})$, $0\leq s<n+1$, analogous to the real case (see Section~\ref{sec:SobolevSpacesCplx} for details). In \cite{MOZ14} we prove that the restriction to $H^{2n-1}=\{(z,t)\in H^{2n+1}:z_n=0\}\subseteq H^{2n+1}$ defines a continuous linear operator $\dot{H}^{s}(H^{2n+1})\to\dot{H}^{s-1}(H^{2n-1})$ whenever $s>1$.

\begin{thmalph}\label{thm:DirichletCplx}
\begin{enumerate}
\item For $-2n<a\leq2$ the operator $\calL_a$ is essentially self-adjoint on the Sobolev type space $\dot{H}^{\frac{2-a}{2}}(H^{2n+1})$.
Its spectrum contains the eigenvalues $2k(2k+a)$ for $k\in\NN$, $0\leq2k<-\frac{a}{2}$.
\item For $-2n<a<0$ and $f\in\dot{H}^{-\frac{a}{2}}(H^{2n-1})$ the Dirichlet problem
\begin{equation}
 \calL_au=0, \qquad u|_{H^{2n-1}}=f\label{eq:CplxBdyValueProblem}
\end{equation}
has a unique $z_n$-radial solution $u\in\dot{H}^{\frac{2-a}{2}}(H^{2n+1})$, i.e. $u(z',z_n,t)=u(z',|z_n|,t)$.
\end{enumerate}
\end{thmalph}

We remark that on functions that are radial in $z_n$ the operator $\calL_a$ acts as
$$ \calL_a = \rho^2\left(\calL'+(a+1)\rho^{-1}\frac{\partial}{\partial\rho}+\frac{\partial^2}{\partial\rho^2}+4\rho^2\frac{\partial^2}{\partial t^2}\right), $$
where $\rho=|z_n|$ and $\calL'$ is the CR-Laplacian on $H^{2n-1}$. The operator in paranthesis on the right hand side is (up to a scaling of the central variable $t$) the operator studied by Frank et al.~\cite{FGMT15}.

In analogy to the real case we also consider the Poisson transform $P_a:\dot{H}^{-\frac{a}{2}}(H^{2n-1})\to\dot{H}^{\frac{2-a}{2}}(H^{2n+1})$ mapping boundary values $f$ to the corresponding $z_n$-radial solutions $u$ of \eqref{eq:CplxBdyValueProblem}.

\begin{thmalph}\label{thm:IsometryCplx}
Assume $-2n<a<0$.
\begin{enumerate}
\item The Poisson transform $P_a$ is the integral operator
$$ P_af(z,t) = c_{n,a} \int_{H^{2n-1}} \frac{|z_n|^{-a}}{|(z,t)^{-1}\cdot(z',0,t')|^{2n-a}} f(z',t') \,d(z',t'), $$
where $c_{n,a}=2^{\frac{2n-a-4}{2}}\pi^{-n}\Gamma(\frac{2n-a}{4})^2\Gamma(-\frac{a}{2})^{-1}$.
\item The operator $P_a:\dot{H}^{-\frac{a}{2}}(H^{2n-1})\to\dot{H}^{\frac{2-a}{2}}(H^{2n+1})$ is isometric up to a constant. More precisely,
$$ \|P_af\|_{\dot{H}^{\frac{2-a}{2}}(H^{2n+1})}^2 = \frac{\pi a}{a-2n}\|f\|_{\dot{H}^{-\frac{a}{2}}(H^{2n-1})}^2. $$
\item For any $1<p\leq\infty$ and $q=\frac{n+1}{n}p$ the Poisson transform $P_a$ extends to a bounded operator $P_a:L^p(H^{2n-1})\to L^q(H^{2n+1})$. More precisely,
$$ \|P_af\|_{L^q(H^{2n+1})} \leq (\pi c_{n,a}^{\frac{1}{n}})^{\frac{1}{q}}\|f\|_{L^p(H^{2n-1})}, \qquad f\in L^p(H^{2n-1}). $$
\end{enumerate}
\end{thmalph}

Also in the Heisenberg case each eigenvalue $2k(2k+a)$ of $\calL_a$
corresponds to a mixed boundary value problem. The corresponding
differential restriction operators
$D_{a,k}:\dot{H}^{\frac{2-a}{2}}(H^{2n+1})\to\dot{H}^{-\frac{a}{2}-2k}(H^{2n-1})$
were first constructed in \cite{MOZ14} and exist also on more general two-step nilpotent groups. In Theorem~\ref{thm:MixedBVPCplx} we study the mixed boundary value problems
$$ \calL_au=2k(2k+a)u, \qquad D_{a,k}u=f, $$
and find their Poisson transforms (see Appendix~\ref{sec:AppendixCplx} for details).

We remark that Frank et al.~\cite{FGMT15} also show an isometry estimate for $P_a$ as in Theorem~\ref{thm:IsometryCplx}~(2), but they neither find the integral kernel of $P_a$, nor do they investigate $L^p$-boundedness of it.

For $p=\frac{4n}{2n+a}$ and $q=\frac{4n+4}{2n+a}$ we know that $\dot{H}^{-\frac{a}{2}}(H^{2n-1})\subseteq L^p(H^{2n-1})$ and $\dot{H}^{\frac{2-a}{2}}(H^{2n+1})\subseteq L^q(H^{2n+1})$ by the Hardy--Littlewood--Sobolev inequality for the Heisenberg group (see e.g. \cite{FS74}). Hence, there is a diagram as in the real case:
$$
\begin{xy}
\xymatrix{
 \dot{H}^{-\frac{a}{2}}(H^{2n-1}) \ar@{^{(}->}[d] \ar[r]^{P_a} & \dot{H}^{\frac{2-a}{2}}(H^{2n+1}) \ar@{^{(}->}[d]\\
 L^p(H^{2n-1}) \ar[r]^{P_a} & L^q(H^{2n+1}).
}
\end{xy}
$$
It is an open question whether Chen's results for the real case have a counterpart in the Heisenberg situation. We formulate this as conjecture:

\begin{conjalph}
There exists a sharp constant $C>0$ such that
$$ \|P_af\|_{L^q(H^{2n+1})} \leq C\|f\|_{L^p(H^{2n-1})} $$
for $p=\frac{4n}{2n+a}$ and $q=\frac{4n+4}{2n+a}$ and the optimizers are translations, dilations and multiples of the function
$$ f(z',t') = ((1+|z'|^2)^2+t'^2)^{-\frac{a+2n}{4}}. $$
\end{conjalph}

\subsection{Relation to representation theory of real reductive groups}

Our proofs use representation theory of the rank one real reductive groups $G=O(1,n+1)$ and $G=U(1,n+1)$. The nilpotent groups $\RR^n$ and $H^{2n+1}=\CC^n\oplus\RR$ occur as nilradical $\overline{N}$ of a parabolic subgroup of $G$ and the homogeneous Sobolev spaces $\dot{H}^s(\overline{N})$ are the Hilbert spaces on which certain irreducible unitary representations $\pi_{-s}$ of $G$ can be realized. These representations are called complementary series and we briefly recall their construction in Section~\ref{sec:CSReps}.

Restricting the representations $\pi_{-s}$ to the subgroup $G'=O(1,n)$ resp. $G'=U(1,n)$ they decompose multiplicity-free into the direct integral of irreducible unitary representations of $G'$. This direct integral has a discrete part if $s>\frac{1}{2}$ in the case of $\overline{N}=\RR^n$ and $s>1$ in the case of $\overline{N}=H^{2n+1}$. One of the discrete summands is precisely the space of solutions $u\in\dot{H}^s(\overline{N})$ to the equation $L_au=0$ for $a=2(1-s)$ where $L_a=\Delta_a$ for $\overline{N}=\RR^n$ and $L_a=\calL_a$ for $\overline{N}=H^{2n+1}$ (allowing only $z_n$-radial solutions in the case $\overline{N}=H^{2n+1}$). The reason for this is that the operator $L_a$ is invariant under the action of $G'$ and hence acts as a scalar on each irreducible summand of $G'$ in the decomposition of $\pi_{-s}|_{G'}$ by Schur's Lemma. In fact, $L_a$ is the Casimir operator of $G'$ in the restricted representation $\pi_{-s}|_{G'}$, and the spectral decomposition of $L_a$ acting on $\dot{H}^s(\overline{N})$ is essentially the decomposition into irreducible $G'$-representations.

This direct summand in the decomposition of $\dot{H}^s(\overline{N})$
consisting of solutions to $L_au=0$ is itself a complementary series
representation of the subgroup $G'$. More precisely, it
is isomorphic to the representation $\tau_{-t}$ of $G'$ on the homogeneous Sobolev space $\dot{H}^{-t}(\overline{N}')$ where $t=s-\frac{1}{2}$ for $\overline{N}'=\RR^{n-1}$ and $t=s-1$ for $\overline{N}'=H^{2n-1}$. The projection onto this direct summand is given by the trace map
$$ \dot{H}^s(\overline{N}) \to \dot{H}^t(\overline{N}'), \quad u\mapsto u|_{\overline{N}'} $$
which is shown to be $G'$-equivariant (also called \textit{$G'$-intertwining}), i.e.
$$ (\pi_{-s}(g)u)|_{\overline{N}'} = \tau_{-t}(g)(u|_{\overline{N}'}) \qquad \forall\,g\in G'. $$
The trace map is a partial isometry and its adjoint is (up to a constant) the corresponding Poisson transform
$$ P_a: \dot{H}^t(\overline{N}') \to \dot{H}^s(\overline{N}), $$
constructing solutions to the boundary value problem
$$ L_au=0, \qquad u|_{\overline{N}'}=f. $$
We remark that the Poisson transform is also $G'$-intertwining, i.e.
$$ P_a(\tau_{-t}(g)f) = \pi_{-s}(g)(P_af) \qquad \forall\,g\in G'. $$

These integral operators from representations of $G'$ to representations of $G$, intertwining the action of $G'$, have recently been investigated for $(G,G')=(O(1,n+1),O(1,n))$ by Kobayashi--Speh~\cite{KS13} and for more general pairs of groups $(G,G')$ by M\"{o}llers--{\O}rsted--Oshima~\cite{MOO13}. More precisely, these works construct the transpose operators, mapping from representations of $G$ to representations of the subgroup $G'$, which are therefore called \textit{symmetry breaking operators}.

We remark that also the $L^p$-$L^q$ boundedness of the Poisson transform has an interpretation in terms of representation theory. In fact, the unitary Hilbert space representations $\pi_{-s}$ on $\dot{H}^s(\overline{N})$ extend to isometric Banach space representations $\widetilde{\pi}_{-s}$ on $L^q(\overline{N})$ where $q=\frac{2n}{n-2s}$ for $\overline{N}=\RR^n$ and $q=\frac{4n+4}{2n+2-s}$ for $\overline{N}=H^{2n+1}$. Extending $\tau_{-t}$ similarly to $\widetilde{\tau}_{-t}$ on $L^p(\overline{N}')$, the Poisson transform provides a bounded $G'$-intertwining embedding of the Banach space representation $(\widetilde{\tau}_{-t},L^p(\overline{N'}))$ into $(\widetilde{\pi}_{-s}|_{G'},L^q(\overline{N}))$

\subsection{Structure of the paper}

In Section~\ref{sec:CSReps} we recall the construction of complementary series representations of the groups $G=O(1,n+1)$ and $G=U(1,n+1)$ and show that they have natural realizations on homogeneous Sobolev spaces. The restriction of complementary series representations of $G$ to the subgroups $G'=O(1,n)$ and $G'=U(1,n)$ is addressed in Section~\ref{sec:SymBreakingOpsVsPoissonTrafos}. Here we summarize the construction of symmetry breaking operators from \cite{KS13,MOO13} and relate them to the Poisson transforms for the boundary value problems \eqref{eq:RealBdyValueProblem} and \eqref{eq:CplxBdyValueProblem}. For particular boundary values we compute in Section~\ref{sec:ExplicitSolutions} the explicit solutions to the boundary value problems, thus finding the right normalization constants for the Poisson transforms. In Sections~\ref{sec:RealCase} and \ref{sec:CplxCase} we finally carry out several computations for the real and complex case separately, such as computing the Casimir operators, showing uniqueness of solutions to the boundary value problems, proving isometry of the Poisson transforms, and establishing the $L^p$-$L^q$ boundedness properties.

More details on the decomposition of the restriction of complementary series representations of $G$ to $G'$ are given in the Appendix. For the real case we indicate in Appendix~\ref{sec:AppendixReal} how the full spectral decomposition of $\Delta_a$ in $\dot{H}^{\frac{2-a}{2}}(\RR^n)$ is obtained. This was carried out in detail by M\"{o}llers--Oshima~\cite{MO12}. For the complex case the full decomposition is not yet know. However, in Appendix~\ref{sec:AppendixCplx} we summarize results of our previous work~\cite{MOZ14} where we construct part of the discrete spectrum of $\calL_a$ in $\dot{H}^{\frac{2-a}{2}}(H^{2n+1})$.

\section{Complementary series representations}\label{sec:CSReps}

We recall the complementary series representations of the rank one groups $G=U(1,n+1;\FF)$, $\FF=\RR,\CC$.

\subsection{Rank one groups}

Let $G=U(1,n+1;\FF)$, $\FF=\RR,\CC$, $n\geq1$, realized as the group of $(n+2)\times(n+2)$ matrices over $\FF$ leaving the sesquilinear form
$$ (x,y) \mapsto x_0\overline{y}_0-x_1\overline{y}_1-\cdots-x_{n+1}\overline{y}_{n+1} $$
invariant. Fix the maximal compact subgroup $K=U(1;\FF)\times U(n+1;\FF)$ which is the fixed point group of the Cartan involution $\theta(g)=(g^*)^{-1}$. Choose
$$ H := \left(\begin{array}{ccc}0&1&\\1&0&\\&&\0_n\end{array}\right), $$
and put $\fraka=\RR H$ and $A=\exp(\fraka)$. Let $\alpha\in\fraka^*$ be such that $\alpha(H)=1$ then $\frakg$ has the grading
$$ \frakg = \frakg_{-2\alpha}\oplus\frakg_{-\alpha}\oplus\frakg_0\oplus\frakg_\alpha\oplus\frakg_{2\alpha}, $$
where $\frakg_\lambda=\{X\in\frakg:[H,X]=\lambda(H)X\}$. Note that $\frakg_{\pm2\alpha}=0$ for $\FF=\RR$. We put $M:=Z_K(\fraka)$ and let $\frakm$ denote its Lie algebra. Then $M=\Delta U(1;\FF)\times U(n;\FF)$ and $\frakg_0=\frakm\oplus\fraka$. Further, let
$$ \frakn := \frakg_\alpha\oplus\frakg_{2\alpha}, \qquad \overline{\frakn}=\theta\frakn = \frakg_{-\alpha}\oplus\frakg_{-2\alpha} $$
and denote by
$$ N := \exp(\frakn), \qquad \overline{N} := \exp(\overline{\frakn}) $$
the corresponding groups. Then $P=MAN$ is a parabolic subgroup of $G$. Let $\rho:=\frac{1}{2}\tr\ad|_{\frakn}\in\fraka^*$.

The element $w_0=\diag(-1,1,\ldots,1)\in K$ represents the non-trivial element in the Weyl group $W=N_K(\fraka)/Z_K(\fraka)$ and acts on $\fraka$ by $-\1$. The parabolic $\overline{P}=\theta P=MA\overline{N}$ opposite to $P$ is conjugate to $P$ via $w_0$, i.e. $w_0Pw_0^{-1}=\overline{P}$.

We identify $\overline{N}\simeq\FF^n\oplus\Im\FF$ by
\begin{equation}
 \FF^n\oplus\Im\FF\to\overline{N}, \quad (z,t)\mapsto n_{(z,t)}:=\exp\left(\begin{array}{ccc}-t/2&-t/2&z^*\\t/2&t/2&-z^*\\z&z&\0_n\end{array}\right),\label{eq:IdentificationNbar}
\end{equation}
where $\Im\FF=\{z\in\FF:z+\overline{z}=0\}$. Under this identification the group multiplication on $\FF^n\oplus\Im\FF$ is given by
$$ (z,t)\cdot(z',t') = (z+z',t+t'+2\Im(z\cdot\overline{z'})). $$
For $(z,t)\in\overline{N}=\FF^n\oplus\Im\FF$ let
$$ |(z,t)| := (|z|^4+|t|^2)^{\frac{1}{4}} $$
denote the norm function. Note that for $\FF=\RR$ we have $\overline{N}\simeq\RR^n$ and for $\FF=\CC$ we have $\overline{N}\simeq H^{2n+1}$.

\subsection{Complementary series representations}

Identify $\fraka_\CC^*\cong\CC$ by $\mu\mapsto\mu(H)$ so that $\rho=\frac{n}{2}$ for $\FF=\RR$ and $\rho=n+1$ for $\FF=\CC$. For $\mu\in\CC$ consider the principal series representations (smooth normalized parabolic induction)
$$ \Ind_P^G(\1\otimes e^\mu\otimes\1) = \{f\in C^\infty(G):f(gman)=a^{-\mu-\rho}f(g)\,\forall\,g\in G,man\in MAN\}. $$
Since $\overline{N}MAN\subseteq G$ is open dense, restriction to $\overline{N}$ realizes these representations as $(\pi_\mu^\infty,I_\mu^\infty)$ with $\calS(\overline{N})\subseteq I_\mu^\infty\subseteq C^\infty(\overline{N})$.

For $\mu\in\RR$ the representation $\pi_\mu^\infty$ is unitarizable if and only if $\mu\in(-\rho,\rho)$. The invariant norm on $I_\mu^\infty$ is for $\FF=\RR$ given by
\begin{equation}
 \|f\|_\mu^2 = \frac{2^{-2\mu}\Gamma(\rho-\mu)}{\pi^{\frac{n}{2}}\Gamma(\mu)} \int_{\overline{N}}\int_{\overline{N}} |n_1n_2^{-1}|^{2(\mu-\rho)} f(n_1)\overline{f(n_2)} \,dn_1 \,dn_2,\label{eq:CSNormReal}
\end{equation}
and for $\FF=\CC$ given by
\begin{equation}
 \|f\|_\mu^2 = \frac{2^{n-1}\Gamma(\frac{\rho+\mu}{2})\Gamma(\frac{\rho-\mu}{2})}{\pi^{n+1}\Gamma(\mu)}\cdot\int_{\overline{N}}\int_{\overline{N}} |n_1n_2^{-1}|^{2(\mu-\rho)} f(n_1)\overline{f(n_2)} \,dn_1 \,dn_2\label{eq:CSNormCplx}
\end{equation}
(or the corresponding regularization of the integral). Note that the norms are normalized such that for $\mu=0$ the norm $\|\blank\|_0$ is equal to the $L^2$-norm on $\RR^n$ resp. $H^{2n+1}$.

Let $I_\mu$ be the completion of $I_\mu^\infty$ with respect to this norm and extend the smooth representation $(\pi_\mu^\infty,I_\mu^\infty)$ to a unitary representation $(\pi_\mu,I_\mu)$. The smooth vectors in this realization are given by $I_\mu^\infty$.

The bilinear pairing
$$ I_\mu^\infty\times I_{-\mu}^\infty\to\CC, \quad (f_1,f_2)\mapsto\int_{\overline{N}}f_1(\overline{n})f_2(\overline{n})\,d\overline{n} $$
is non-degenerate and $G$-invariant for
$\pi_\mu^\infty\times\pi_{-\mu}^\infty$. Hence we can identify
$I_\mu^\infty$ with a subrepresentation of the dual representation
$(I_{-\mu}^\infty)^*$. We call
$(\pi_\mu^{-\infty},I_\mu^{-\infty}):=((\pi_{-\mu}^\infty)^*,(I_{-\mu}^\infty)^*)$
the \textit{distribution globalization} 
of $(\pi_\mu^\infty,I_\mu^\infty)$ since
$$ \calE'(\overline{N})\subseteq I_\mu^{-\infty}\subseteq\calS'(\overline{N}). $$
Altogether we obtain embeddings
\begin{equation}
 \calS(\overline{N})\subseteq I_\mu^\infty\subseteq I_\mu^{-\infty}\subseteq\calS'(\overline{N}).\label{eq:DistributionVectorsEmbeddings}
\end{equation}
If now $\mu\in(-\rho,\rho)$ then $I_\mu^{-\infty}$ is the space of distribution vectors of $I_\mu$ and we have the following embeddings of representations:
\begin{equation}
 (\pi_\mu^\infty,I_\mu^\infty) \subseteq (\pi_\mu,I_\mu) \subseteq (\pi_\mu^{-\infty},I_\mu^{-\infty}).\label{eq:DistributionVectorsEmbeddingsUnitary}
\end{equation}

\subsection{Relation to homogeneous Sobolev spaces}

We explain how the Hilbert spaces $I_\mu$, $\mu\in[0,\rho)$, can be viewed as homogeneous Sobolev spaces on the nilpotent group $\overline{N}$. For this we consider the two cases $\FF=\RR$ and $\FF=\CC$ separately.

\subsubsection{The real case}\label{sec:SobolevSpacesReal}

Consider the Euclidean Fourier transform
$$ \calF_{\RR^n}:\calS'(\RR^n)\to\calS'(\RR^n), \quad \calF_{\RR^n}u(\xi) = \widehat{u}(\xi) = (2\pi)^{-\frac{n}{2}}\int_{\RR^n} e^{-ix\cdot\xi} u(x) \,dx. $$
Using the Plancherel formula
$$ \int_{\RR^n} u(x)\overline{v(x)}\,dx = \int_{\RR^n} \widehat{u}(\xi)\overline{\widehat{v}(\xi)}\,d\xi $$
and the convolution formula
$$ \widehat{u*v}(\xi) = (2\pi)^{\frac{n}{2}}\widehat{u}(\xi)\widehat{v}(\xi) $$
we can express the norm $\|\blank\|_\mu$ on $I_\mu$ defined by \eqref{eq:CSNormReal} as
\begin{align*}
 \|u\|_\mu^2 &= \frac{2^{-2\mu}\Gamma(\frac{n}{2}-\mu)}{\pi^{\frac{n}{2}}\Gamma(\mu)} \int_{\RR^n} (|\blank|^{2\mu-n}*u)(x)\overline{u}(x)\,dx\\
 &= \frac{2^{-2\mu+\frac{n}{2}}\Gamma(\frac{n}{2}-\mu)}{\Gamma(\mu)} \int_{\RR^n} |\widehat{u}(\xi)|^2 \widehat{|\blank|^{2\mu-n}}(\xi)\,d\xi.
\end{align*}
Now, by \cite[Chapter II, Section 3.3]{GS64} we have
$$ \widehat{|\blank|^\lambda} = \frac{2^{\lambda+\frac{n}{2}}\Gamma(\frac{\lambda+n}{2})}{\Gamma(-\frac{\lambda}{2})} |\blank|^{-\lambda-n} $$
and hence
$$ \|u\|_\mu^2 = \int_{\RR^n} |\widehat{u}(\xi)|^2 |\xi|^{-2\mu}\,d\xi. $$
Therefore $I_\mu=\dot{H}^{-\mu}(\RR^n)$ for $\mu\in(-\frac{n}{2},\frac{n}{2})$, where
$$ \dot{H}^s(\RR^n) = \left\{u\in\calS'(\RR^n):\widehat{u}\in L^2_{\text{loc}}(\RR^n),\,\|u\|_{\dot{H}^s(\RR^n)}^2:=\int_{\RR^n}|\widehat{u}(\xi)|^2|\xi|^{2s}\,d\xi<\infty\right\} $$
is the \textit{homogeneous Sobolev space of degree $s$}. Note that by our previous calculations
\begin{equation}
 \|u\|_\mu = \|u\|_{\dot{H}^{-\mu}(\RR^n)} = \|\widehat{u}\|_{L^2(\RR^n,|\xi|^{-2\mu}\,d\xi)}.\label{eq:IsometryFourierTrafoSobolevL2}
\end{equation}

\subsubsection{The complex case}\label{sec:SobolevSpacesCplx}

For $\mu\in\RR^\times$ let
$$ \calF_\mu := \left\{\xi\in\calO(\CC^n):\|\xi\|_\mu^2=\int_{\CC^n}|\xi(w)|^2e^{-2|\mu||w|^2}\,dw<\infty\right\} $$
and define an irreducible unitary representation $\sigma_\mu$ of $H^{2n+1}$ on $\calF_\mu$ by
$$ \sigma_\mu(z,t)\xi(w) := \begin{cases}e^{i\mu t+2\mu(w\cdot z-|z|^2/2)}\xi(w-\overline{z}) & \mbox{for $\mu>0$,}\\e^{i\mu t+2\mu(w\cdot\overline{z}+|z|^2/2)}\xi(w+z) & \mbox{for $\mu<0$.}\end{cases} $$
For $u\in L^1(H^{2n+1})$ let
$$ \sigma_\mu(u) := \int_{H^{2n+1}} u(z,t)\sigma_\mu(z,t) \,d(z,t) $$
denote the group Fourier transform. Then $\sigma_\mu$ extends to $L^2(H^{2n+1})$ and we have the Plancherel formula
$$ \|u\|_{L^2(H^{2n+1})}^2 = \frac{2^{n-1}}{\pi^{n+1}} \int_\RR \|\sigma_\mu(u)\|_{\HS(\calF_\mu)}^2 |\mu|^n \,d\mu, $$
and the inversion formula
$$ u(z,t) = \frac{2^{n-1}}{\pi^{n+1}} \int_\RR \tr(\sigma_\mu(z,t)^*\sigma_\mu(u)) |\mu|^n \,d\mu, $$
where $\|T\|_{\HS(\calF_\mu)}^2=\tr(T^*T)$ denotes the Hilbert--Schmidt norm of an operator $T$ on $\calF_\mu$.

The space $\calP$ of polynomials on $\CC^n$ is dense in $\calF_\mu$ and we write $\calP_m$ for its subspace of homogeneous polynomials of degree $m$. The subspaces $\calP_m$ are pairwise orthogonal and we denote by $P_m:\calF_\mu\to\calP_m$ the orthogonal projections. Then we can write the Hilbert--Schmidt norm as
$$ \|T\|_{\HS(\calF_\mu)}^2 = \sum_{m=0}^\infty \|P_m\circ T\|_{\HS(\calF_\mu)}^2. $$
and hence, the $L^2$-norm of a function $u\in L^2(H^{2n+1})$ is given by
$$ \|u\|_{L^2(H^{2n+1})}^2 = \frac{2^{n-1}}{\pi^{n+1}} \sum_{m=0}^\infty \int_\RR \|P_m\circ\sigma_\mu(u)\|_{\HS(\calF_\mu)}^2 |\mu|^n \,d\mu. $$
For $s\in(-n-1,n+1)$ we define a new norm $\|\blank\|_{\dot{H}^s(H^{2n+1})}$ on $C_c^\infty(H^{2n+1})$ by
\begin{equation}
 \|u\|_{\dot{H}^s(H^{2n+1})}^2 = \frac{2^{n-1}}{\pi^{n+1}} \sum_{m=0}^\infty \frac{(\frac{n+1+s}{2})_m}{(\frac{n+1-s}{2})_m} \int_\RR \|P_m\circ\sigma_\mu(u)\|_{\HS(\calF_\mu)}^2 |\mu|^{n+s} \,d\mu.\label{eq:SobolevNormCplx}
\end{equation}
The completion of $C_c^\infty(H^{2n+1})$ with respect to the norm $\|\blank\|_{\dot{H}^s(H^{2n+1})}$ will be denoted by $\dot{H}^s(H^{2n+1})$ and is called \textit{homogeneous Sobolev space of degree $s$}.

In his paper \cite[Theorem 8.1]{Cow82} Cowling showed that the norm $\|\blank\|_\mu$ defined by \eqref{eq:CSNormCplx} is equal to the norm $\|\blank\|_{\dot{H}^{-\mu}(H^{2n+1})}$:
$$ \|u\|_\mu = \|u\|_{\dot{H}^{-\mu}(H^{2n+1})}. $$
Hence the Hilbert space $I_\mu$ is equal to the homogeneous Sobolev space $\dot{H}^{-\mu}(H^{2n+1})$.

\section{Symmetry breaking operators vs. Poisson transforms}\label{sec:SymBreakingOpsVsPoissonTrafos}

In this section we explain how the recently constructed symmetry breaking operators between induced representations (see \cite{KS13,MOO13} for details) can be used to construct Poisson transforms for certain boundary value problems on the nilpotent groups $\overline{N}$.

\subsection{Symmetric pairs}

Let $G'=G^\sigma$ be the fixed point group of the involution
$$ \sigma(g) = \1_{n+1,1}\cdot g\cdot \1_{n+1,1} $$
given by conjugation with the matrix $\1_{n+1,1}=\diag(1,\ldots,1,-1)$. In the canonical block diagonal decomposition we can identify $G'\simeq U(1,n;\FF)\times U(1;\FF)$. By definition $(G,G')$ is a symmetric pair, and we note that $P'=P\cap G'$ is a parabolic subgroup of $G'$. Its Langlands decomposition is given by $P'=M'AN'$ with $M'=\Delta U(1;\FF)\times U(n-1;\FF)\times U(1;\FF)$ and $N'=N\cap H$. Identifying $\overline{N}\simeq\FF^n\oplus\Im\FF$ as before the subgroup $\overline{N}'=\theta N'$ is identified with the subspace $\FF^{n-1}\oplus\Im\FF\subseteq\FF^n\oplus\Im\FF$ where $\FF^{n-1}\subseteq\FF^n$ as the first $n-1$ coordinates.

For $\nu\in\CC$ consider the principal series representations
$$ \Ind_{P'}^{G'}(\1\otimes e^\nu\otimes\1) $$
of $G'$. Again we realize these representations on smooth functions on
$\overline{N}'\simeq\FF^{n-1}\oplus\Im\FF$ and denote this realization by
$(\tau_\nu^\infty,J_\nu^\infty)$ with $\calS(\overline{N}')\subseteq
J_\nu^\infty\subseteq C^\infty(\overline{N}')$. Denote by $(\tau_\nu,J_\nu)$
the unitary globalization of $(\tau_\nu^\infty,J_\nu^\infty)$ whenever $\tau_\nu^\infty$ is unitarizable, and by $(\tau_\nu^{-\infty},J_\nu^{-\infty})$ the distribution globalization.

\subsection{Symmetry breaking operators}\label{sec:SymBreakingOps}

In \cite{KS13} and \cite{MOO13} a meromorphic family of $G'$-intertwining operators $A_{\mu,\nu}:I_\mu^\infty\to J_\nu^\infty$ is constructed as singular integral operators
$$ A_{\mu,\nu}u(n') = \int_{\overline{N}} K_{\mu,\nu}(n,n')u(n)\,dn, \qquad n'\in\overline{N}'. $$
The kernel $K_{\mu,\nu}(n,n')$ on $\overline{N}\times\overline{N}'$ is given by
$$ K_{\mu,\nu}((z,t),(z',t')) = \frac{|z_n|^{(\mu-\rho)+(\nu+\rho')}}{|(z,t)^{-1}\cdot(z',0,t')|^{2(\nu+\rho')}}, \qquad (z,t)\in\overline{N},(z',t)\in\overline{N}', $$
where $z=(z_1,\ldots,z_n)\in\FF^n$ and $\rho'=\frac{1}{2}\tr\ad|_{\frakn'}$. The intertwining property of these operators can be written as
$$ A_{\mu,\nu}\circ\pi_\mu^\infty(g) = \tau_\nu^\infty(g)\circ A_{\mu,\nu} \qquad \forall\,g\in G'. $$

The transpose operator $A_{\mu,\nu}^T:J_\nu^*\to I_\mu^*$ can in view of \eqref{eq:DistributionVectorsEmbeddings} be interpreted as an intertwining operator between the distribution globalizations $\tau_{-\nu}^{-\infty}$ and $\pi_{-\mu}^{-\infty}$. We put
$$ B_{\mu,\nu} := A_{-\mu,-\nu}^T: J_\nu^{-\infty}\to I_\mu^{-\infty} $$
which is given by
$$ B_{\mu,\nu}f(n) = \int_{\overline{N}} K_{-\mu,-\nu}(n,n')f(n')\,dn', \qquad n\in\overline{N}, $$
and satisfies the intertwining property
$$ B_{\mu,\nu}\circ\tau_\nu^{-\infty}(g) = \pi_\mu^{-\infty}(g)\circ B_{\mu,\nu} \qquad \forall\,g\in G'. $$

\subsection{The Casimir operator}\label{sec:CasimirOperator}

Define
$$ B(X,Y) := \frac{1}{2}\Re\tr(XY), \qquad X,Y\in\frakg. $$
Then $B$ is a non-degenerate invariant bilinear form on $\frakg$, the Lie algebra of $G$. Write $G'=G_1'\times G_2'$ with $G_1'=U(1,n;\FF)$ and $G_2'=U(1;\FF)$. Consider the Casimir elements $C_1,C_2\in\calU(\frakg)^{G'}$ of $G_1'$ and $G_2'$ within the universal enveloping algebra of $\frakg$ with respect to this form. (Note that $G_2'=O(1)$ for $\FF=\RR$ and hence $\frakg_2'=0$ which implies $C_2=0$.) We study the action of the element
$$ C:=C_1-C_2\in\calU(\frakg)^{G'} $$
in the representation $\pi_\mu^{-\infty}$. For this denote by $d\pi_\mu^{-\infty}$ the differentiated representation of $\calU(\frakg)$ on $I_\mu^{-\infty}$.

To state the result we identify $X\in\overline{\frakn}$ with the left-invariant differential operator on $\overline{N}$ given by
$$ (Xf)(\overline{n}) = \left.\frac{d}{ds}\right|_{s=0}f(\overline{n}e^{sX}), \qquad \overline{n}\in\overline{N}. $$
Let $X_\alpha$ be an orthonormal basis of $\FF^n\subseteq\overline{N}$ with respect to the standard inner product on $\FF^n$ and put
$$ Lf(z,t) := \sum_\alpha X_\alpha^2f(z,t), \qquad E_nf(z,t) := \left.\frac{d}{ds}\right|_{s=0}f(z',z_n+sz_n,t). $$
Then $E_n$ is the Euler operator in the coordinate $z_n$ and $L$ is for $\FF=\RR$ the usual Laplacian $\Delta$ and for $\FF=\CC$ the CR-Laplacian $\calL$.

\begin{proposition}\label{prop:ActionCasimir}
For $\mu\in\CC$ we have
$$ d\pi_\mu^{-\infty}(C) = |z_n|^2L+2(\mu+1)E_n+(\mu+\rho)(\mu+\rho-2\rho'). $$
\end{proposition}

The proof will be given separately for the two cases $\FF=\RR,\CC$ in Sections~\ref{sec:CasimirReal} and \ref{sec:CasimirCplx}.

\begin{corollary}\label{cor:SelfAdjoint}
Assume $\mu\in(-\rho,\rho)$ then the operator
$$ |z_n|^2L+2(\mu+1)E_n $$
is essentially self-adjoint in the Hilbert space $I_\mu$.
\end{corollary}

\begin{proof}
In a unitary representation the Casimir element defines a self-adjoint operator by \cite[Theorem 4.4.4.3]{War72} and hence the statement is clear by Proposition~\ref{prop:ActionCasimir}.
\end{proof}

\begin{corollary}\label{cor:SolutionsOnDistributions}
The image of the operator $B_{\mu,\nu}:J_\nu^{-\infty}\to I_\mu^{-\infty}$ consists of solutions of the following differential equation:
$$ (|z_n|^2L+2(\mu+1)E_n)u = \lambda u, $$
where $\lambda=-((\mu+\rho)-(\nu+\rho'))((\mu+\rho)+(\nu-\rho'))$.
\end{corollary}

\begin{proof}
Since the representation $J_\nu^{-\infty}$ has infinitesimal character $\nu+\rho'$ the Casimir element $C_1$ acts by $\nu^2-\rho'^2$. Further $C_2$ acts trivially and hence $d\tau_\nu^{-\infty}(C)$ is the scalar $\nu^2-\rho'^2$. Since $B_{\mu,\nu}$ is intertwining we obtain
$$ d\pi_\mu^{-\infty}(C)\circ B_{\mu,\nu} = B_{\mu,\nu}\circ d\tau_\nu^{-\infty}(C) = (\nu^2-\rho'^2)B_{\mu,\nu} $$
and the claim follows from Proposition~\ref{prop:ActionCasimir}.
\end{proof}

\subsection{Trace maps}\label{sec:TraceMaps}

We consider the trace map $T$ which restricts functions on $\overline{N}\simeq\FF^n\oplus\Im\FF$ to the subgroup $\overline{N}'\simeq\FF^{n-1}\oplus\Im\FF$:
$$ Tf(z',t) := f(z',0,t), \qquad (z',t)\in\overline{N}'. $$

\begin{lemma}\label{lem:ResIntertwining}
For $\mu+\rho=\nu+\rho'$ the operator $T$ maps $I_\mu^\infty$ to $J_\nu^\infty$ and intertwines the representations $\pi_\mu^\infty|_{G'}$ and $\tau_\nu^\infty$, i.e.
$$ T\circ\pi_\mu^\infty(g) = \tau_\nu^\infty(g)\circ T \qquad \forall\,g\in G'. $$
\end{lemma}

\begin{proof}
Consider the restriction operator $\widetilde{T}:C^\infty(G)\to C^\infty(G')$. It is clear from the definition of the induced representations that $\widetilde{T}$ maps $\Ind_P^G(\1\otimes e^\mu\otimes\1)$ to $\Ind_{P'}^{G'}(\1\otimes e^\nu\otimes\1)$
and intertwines the corresponding actions of $G'$. Hence, we have the following commutative diagram where the vertical arrows are restriction to $\overline{N}$ and $\overline{N}'$:
$$
\begin{xy}
\xymatrix{
 \Ind_P^G(\1\otimes e^\mu\otimes\1) \ar[d]^{\sim} \ar[r]^{\widetilde{T}} & \Ind_{P'}^{G'}(\1\otimes e^\nu\otimes\1) \ar[d]^{\sim}\\
 I_\mu^\infty \ar[r]^{T} & J_\nu^\infty.
}
\end{xy}
$$
Since the vertical arrows are bijections this shows the claim.
\end{proof}

The following result is standard for $\FF=\RR$ and proved in \cite[Theorem 4.6]{MOZ14} for $\FF=\CC$.

\begin{proposition}
The trace map $T$ extends to a continuous linear operator
$$ T:\dot{H}^s(\overline{N})\to\dot{H}^{s-\frac{d}{2}}(\overline{N}') $$
for every $s>\frac{d}{2}$, where $d=\dim_\RR\FF$.
\end{proposition}

\begin{corollary}
Assume $s\in(\frac{d}{2},\rho)$ then the adjoint $T^*:\dot{H}^{s-\frac{d}{2}}(\overline{N}')\to\dot{H}^s(\overline{N})$ is an isometry (up to scalar) and $TT^*$ is a scalar multiple of the identity on $\dot{H}^{s-\frac{d}{2}}(\overline{N}')$.
\end{corollary}

\begin{proof}
Let $\mu=-s$ and $\nu=-s+\frac{d}{2}$ then $\mu+\rho=\nu+\rho'$ and hence, by Lemma~\ref{lem:ResIntertwining}, the restriction operator $T:\dot{H}^s(\overline{N})\to\dot{H}^{s-\frac{d}{2}}(\overline{N}')$ is intertwining $\pi_\mu|_{G'}\to\tau_\nu$. This clearly implies that the adjoint is intertwining $\tau_\nu\to\pi_\mu|_{G'}$. Therefore $TT^*$ is an endomorphism of the irreducible unitary representation $(\tau_\nu,\dot{H}^{s-\frac{d}{2}}(\overline{N}'))$ and hence a scalar multiple of the identity by Schur's Lemma, say $TT^*=c\cdot\id$. Then
$$ \|T^*f\|_{\dot{H}^s(\overline{N})}^2 = \langle T^*f,T^*f\rangle_{\dot{H}^s(\overline{N})} = \langle TT^*f,f\rangle_{\dot{H}^{s-\frac{d}{2}}(\overline{N}')} = c\|f\|_{\dot{H}^{s-\frac{d}{2}}(\overline{N}')}^2 $$
which finishes the proof.
\end{proof}

In the next step we show that $T^*$ is a scalar multiple of an operator of the form $B_{\mu,\nu}$, the transpose of a symmetry breaking operator introduced in Section~\ref{sec:SymBreakingOps}. For this we make use of a strong representation theoretic result referred to as \textit{Multiplicity One Theorem}.

\begin{fact}[{see \cite{SZ12}}]\label{fct:MultOneThms}
Let $G=O(1,n+1)$ resp. $U(1,n+1)$ and $G'=O(1,n)$ resp. $U(1,n)$. Then for any irreducible Casselman--Wallach representations $\pi$ of $G$ and $\tau$ of $G'$ the space of $G'$-intertwining operators $\tau\to(\pi^*)|_{G'}$ is at most one-dimensional.
\end{fact}

Here a representation $\pi$ of $G$ is called \textit{Casselman--Wallach} if it is a smooth representation on a Fr\'{e}chet space which is admissible, of moderate growth and finite under the center of the universal enveloping algebra. We note that the representations $(\pi_\mu^\infty,I_\mu^\infty)$ of $G$ and $(\tau_\nu^\infty,J_\nu^\infty)$ of $G'$ are Casselman--Wallach.

\begin{theorem}\label{thm:PoissonTransformsAsSymBreakingOps}
Let $\frac{d}{2}<s<\rho$ and put $\mu=-s$ and $\nu=-s+\frac{d}{2}$. Then $T^*$ is a scalar multiple of the symmetry breaking operator $B_{\mu,\nu}$. In particular, $B_{\mu,\nu}$ restricts to an isometry $\dot{H}^{s-\frac{d}{2}}(\overline{N}')\to\dot{H}^s(\overline{N})$ (up to scalar) and $T\circ B_{\mu,\nu}$ is a scalar multiple of the identity.
\end{theorem}

\begin{proof}
Using the embeddings \eqref{eq:DistributionVectorsEmbeddings} we consider both $T^*$ and $B_{\mu,\nu}$ as operators $J_\nu^\infty\to I_\mu^{-\infty}$, intertwining the representations $\tau_\nu^\infty$ and $\pi_\mu^{-\infty}|_{G'}=(\pi_{-\mu}^\infty)^*|_{G'}$. Note that the representations $(\tau_\nu^\infty,J_\nu)$ of $G'$ and $(\pi_{-\mu}^\infty,I_{-\mu}^\infty)$ of $G$ are irreducible. Hence, by Fact~\ref{fct:MultOneThms} the operators $T^*$ and $B_{\mu,\nu}$ are proportional, showing the claim.
\end{proof}

\subsection{Poisson transforms}

Let us collect the results of Corollaries~\ref{cor:SelfAdjoint}, \ref{cor:SolutionsOnDistributions} and Theorem~\ref{thm:PoissonTransformsAsSymBreakingOps}. For $a\in\RR$ put
$$ L_a := |z_n|^2L+aE_n, $$
so that $d\pi_\mu^\infty(C)=L_a+(\mu+\rho)(\mu+\rho-2\rho')$ for $\mu=\frac{a-2}{2}$. Then Corollary~\ref{cor:SelfAdjoint} implies that $L_a$ is self-adjoint on $I_\mu=\dot{H}^{\frac{2-a}{2}}(\overline{N})$ whenever $\mu\in(-\rho,\rho)$ or equivalently $2-2\rho<a<2+2\rho$. This proves the self-adjointness statements of Theorem~\ref{thm:DirichletReal}~(1) and Theorem~\ref{thm:DirichletCplx}~(1). The statements about the spectrum of $L_a$ on $\dot{H}^{\frac{2-a}{2}}(\overline{N})$ are shown in Appendix~\ref{sec:AppendixReal} and \ref{sec:AppendixCplx}.

Next we consider the Dirichlet problem
\begin{equation}
 L_au=0, \qquad u|_{\overline{N}'}=f,\label{eq:DirichletProblem}
\end{equation}
where restriction $u|_{\overline{N}'}$ should be interpreted in terms of the trace operator (see Section~\ref{sec:TraceMaps}). Let $\nu$ be such that $\mu+\rho=\nu+\rho'$. Then Corollary~\ref{cor:SolutionsOnDistributions} and
Theorem~\ref{thm:PoissonTransformsAsSymBreakingOps} imply that
there exists a constant $c_{n,a}$ such that the map
$$ P_a := c_{n,a}B_{\mu,\nu}:\dot{H}^{\frac{2-a-d}{2}}(\overline{N}')=J_\nu\to I_\mu=\dot{H}^{\frac{2-a}{2}}(\overline{N}) $$
is a solution operator for the Dirichlet problem, i.e. for $f\in \dot{H}^{\frac{2-a-d}{2}}(\overline{N}')$ the function $u=P_af\in\dot{H}^{\frac{2-a}{2}}(\overline{N})$ solves \eqref{eq:DirichletProblem}. Further, $P_a$ is (up to a constant) an isometry $\dot{H}^{\frac{2-a-d}{2}}(\overline{N}')\to\dot{H}^{\frac{2-a}{2}}(\overline{N})$ and has the integral representation
\begin{equation*}
 P_af(n) = c_{n,a} \int_{\overline{N}'} K_{-\mu,-\nu}(n,n')f(n')\,dn'.\label{eq:GenIntRepPoissonTrafo}
\end{equation*}
This proves parts of Theorems~\ref{thm:DirichletReal}--\ref{thm:IsometryCplx}, leaving open the following:
\begin{itemize}
\item Uniqueness of the solutions (see Sections~\ref{sec:UniquenessReal} and \ref{sec:UniquenessCplx}),
\item Computation of the constant $c_{n,a}$ (see Sections~\ref{sec:ExplicitInvBdyValuesReal} and \ref{sec:ExplicitInvBdyValuesCplx}),
\item Computation of the constant in the isometry identities (see Sections~\ref{sec:IsometryReal} and \ref{sec:IsometryCplx}),
\item $L^p$-$L^q$ boundedness (see Sections~\ref{sec:LpReal} and \ref{sec:LpCplx}).
\end{itemize}

\newpage

\section{Solutions for explicit boundary values}\label{sec:ExplicitSolutions}

In this section we explicitly calculate the Poisson transforms for certain boundary values.

\subsection{The real case}

\subsubsection{$K'$-invariant boundary values}\label{sec:ExplicitInvBdyValuesReal}

Let $f(y)=(1+|y|^2)^{-\frac{a+n-2}{2}}$, the $K'$-invariant vector in the representation $(\tau_\nu,J_\nu)$ for $\nu=\frac{a-1}{2}$. 

\begin{lemma}
We have $f\in\dot{H}^{\frac{1-a}{2}}(\RR^{n-1})$ and for $x\in\RR^n$:
\begin{align*}
 P_af(x) &= \frac{2^{\frac{a+n-2}{2}}\pi^{\frac{n-1}{2}}\Gamma(\frac{n-1}{2})}{\Gamma(n-1)}c_{n,a}\cdot{_2F_1}\left(\frac{a+n-2}{4},\frac{a+n}{4};\frac{n}{2};1-\frac{4x_n^2}{(1+|x|^2)^2}\right)(1+|x|^2)^{\frac{2-a-n}{2}}.
\end{align*}
\end{lemma}

\begin{proof}
The function $f\in\dot{H}^{\frac{1-a}{2}}(\RR^{n-1})$ is $K'$-invariant under $\tau_\nu$ for $\nu=\frac{a-1}{2}$. Since $P_a$ intertwines $\tau_\nu$ and $\pi_\mu$ for $\mu=\frac{a-2}{2}$ the Poisson transform $P_af\in\dot{H}^{\frac{2-a}{2}}(\RR^n)$ of $f$ has to be $K'$-invariant under $\pi_\mu$. Now, the stereographic projection induces an isomorphism
$$ \calV: C^\infty(S^n) \to I_\mu^\infty, \quad \calV u(x)=(1+|x|^2)^{-\mu-\rho}u\left(\frac{1-|x|^2}{1+|x|^2},\frac{2x_1}{1+|x|^2},\ldots,\frac{2x_n}{1+|x|^2}\right), $$
under which the action $\pi_\mu^\infty$ of $K=O(1)\times O(n+1)$ on $I_\mu^\infty$ corresponds to the action of $K$ on $C^\infty(S^n)$ induced from the natural action of $K$ on $S^n$. Therefore, the function $\calV^{-1}P_af$ on $S^n$ which is invariant under $K'=O(1)\times O(n)\times O(1)$ only depends on the last coordinate, and hence
$$ P_af(x) = (1+|x|^2)^{\frac{2-a-n}{2}}F\left(\frac{2x_n}{1+|x|^2}\right) $$
for some function $F$ on $[-1,1]$. Putting $x'=0$ in the integral representation \eqref{eq:GenIntRepPoissonTrafo} we find 
\begin{align*}
 P_af(0,x_n) &= c_{n,a}\int_{\RR^{n-1}} \frac{|x_n|^{1-a}}{(|y|^2+x_n^2)^{\frac{n-a}{2}}}(1+|y|^2)^{\frac{2-a-n}{2}} \,dy\\
 &= \frac{2\pi^{\frac{n-1}{2}}}{\Gamma(\frac{n-1}{2})}c_{n,a}|x_n|^{1-n} \int_0^\infty (1+|x_n|^{-2}r^2)^{\frac{a-n}{2}}(1+r^2)^{\frac{2-a-n}{2}}r^{n-2} \,dr\\
 &= \frac{\pi^{\frac{n-1}{2}}B(\frac{n-1}{2},\frac{n-1}{2})}{\Gamma(\frac{n-1}{2})}c_{n,a}\cdot{_2F_1}\left(\frac{a+n-2}{2},\frac{n-1}{2};n-1;1-x_n^2\right)\\
 &= \frac{\pi^{\frac{n-1}{2}}\Gamma(\frac{n-1}{2})}{\Gamma(n-1)}c_{n,a}\cdot{_2F_1}\left(\frac{a+n-2}{4},\frac{a+n}{4};\frac{n}{2};\left(\frac{1-x_n^2}{1+x_n^2}\right)^2\right)\left(\frac{1+x_n^2}{2}\right)^{\frac{2-a-n}{2}},
\end{align*}
where we have used the integral formula \cite[formula 3.259~(3)]{GR07} and the transformation formula \cite[formula 9.134~(1)]{GR07}. This implies
$$ F(z) = \frac{2^{\frac{a+n-2}{2}}\pi^{\frac{n-1}{2}}\Gamma(\frac{n-1}{2})}{\Gamma(n-1)}c_{n,a}\cdot{_2F_1}\left(\frac{a+n-2}{4},\frac{a+n}{4};\frac{n}{2};1-z^2\right) $$
and the proof is complete.
\end{proof}

We can now compute the constant $c_{n,a}$ by putting $x_n=0$ in the above formula for $P_af(x)$. By \cite[Theorem 2.2.2]{AAR99} and the duplication formula for the gamma function we have
$$ P_af(y,0) = \frac{\pi^{\frac{n-1}{2}}\Gamma(\frac{1-a}{2})}{\Gamma(\frac{n-a}{2})}c_{n,a}\cdot f(y). $$
Since $P_af|_{\RR^{n-1}}=f$ we must have
$$ c_{n,a} = \frac{\Gamma(\frac{n-a}{2})}{\pi^{\frac{n-1}{2}}\Gamma(\frac{1-a}{2})}. $$

\begin{remark}
For $a=0$ the differential operator $\Delta_a=|x_n|^2\Delta$ is the Laplacian times $|x_n|^2$. In this case the function $P_af$ simplifies to
$$ P_af(x) = \begin{cases}|x+e_n|^{2-n} & \mbox{for $x_n\geq0$,}\\|x-e_n|^{2-n} & \mbox{for $x_n\leq0$,}\end{cases} $$
which follows from the identity
\begin{align*}
 {_2F_1}\left(\alpha,\alpha+\frac{1}{2};2\alpha;z\right) &= \frac{2^{2\alpha-1}(\sqrt{1-z}+1)^{1-2\alpha}}{\sqrt{1-z}}.
\end{align*}
Note that up to translation $P_af$ agrees on the upper half-space $\{x_n>0\}$ and the lower half-space $\{x_n<0\}$ with the fundamental solution for the Laplacian. At the boundary $\{x_n=0\}$ the function $P_af$ is continuous, but not differentiable:
$$ \singsupp(P_af) = \{(x',0):x'\in\RR^{n-1}\}. $$
\end{remark}

\subsubsection{Constant boundary values}

Let $\1_{\RR^k}$ denote the constant function with value $1$ on $\RR^k$.

\begin{lemma}\label{lem:ExplicitBVConstantReal}
Let $n\geq2$. For $a<1$ we have
$$ P_a\1_{\RR^{n-1}} = \1_{\RR^n}. $$
\end{lemma}

\begin{proof}
We compute
\begin{align*}
 P_a\1_{\RR^{n-1}}(x) &= c_{n,a} \int_{\RR^{n-1}} \frac{|x_n|^{1-a}}{(|x'-y|^2+x_n^2)^{\frac{n-a}{2}}}\,dy\\
 &= c_{n,a}|x_n|^{1-a} \int_{\RR^{n-1}} (|y|^2+x_n^2)^{\frac{a-n}{2}}\,dy\\
 &= \frac{2c_{n,a}\pi^{\frac{n-1}{2}}}{\Gamma(\frac{n-1}{2})}|x_n|^{1-a} \int_0^\infty (r^2+x_n^2)^{\frac{a-n}{2}} r^{n-2}\,dr.
\end{align*}
By the integral formula \cite[formula 3.251~(2)]{GR07} we have
$$ \int_0^\infty (r^2+x_n^2)^{\frac{a-n}{2}} r^{n-2}\,dr = \frac{1}{2}|x_n|^{a-1}B\left(\frac{n-1}{2},\frac{1-a}{2}\right) $$
and the claimed identity follows.
\end{proof}

\newpage

\subsection{The complex case}

\subsubsection{$K'$-invariant boundary values}\label{sec:ExplicitInvBdyValuesCplx}

Let $f(z',t')=((1+|z'|^2)^2+t'^2)^{-\frac{a+2n}{4}}$.

\begin{lemma}
We have $f\in\dot{H}^{-\frac{a}{2}}(H^{2n-1})$ and
\begin{multline*}
 P_af(x) = \frac{2^{\frac{a+2n}{2}}\pi^{n-\frac{1}{2}}\Gamma(n-\frac{1}{2})}{\Gamma(2n-1)}c_{n,a}\cdot{_2F_1}\left(\frac{a+2n}{4},\frac{a+2n}{4};n;1-\frac{4|z_n|^2}{(1+|z|^2)^2+t^2}\right)\\
 \times((1+|z|^2)^2+t^2)^{-\frac{a+2n}{4}}.
\end{multline*}
\end{lemma}

\begin{proof}
Similar to the proof in the real case, using the stereographic projection for the Heisenberg group
$$ H^{2n+1} \to S^{2n+1}\subseteq\CC^{n+1}, \quad (z,t)\mapsto\left(\frac{(1-|z|^2)-it}{(1+|z|^2)+it},\frac{2z_1}{(1+|z|^2)+it},\ldots,\frac{2z_n}{(1+|z|^2)+it}\right) $$
we find that the function $u=P_af\in\dot{H}^{\frac{2-a}{2}}(H^{2n+1})$ which is invariant under the action $\pi_\mu$ of $K'=U(1)\times U(n)\times U(1)$ has to be of the form
$$ u(z,t) = ((1+|z|^2)^2+t^2)^{-\frac{a+2n}{4}}F\left(\frac{4|z_n|^2}{(1+|z|^2)^2+t^2}\right) $$
for some function $F$ on $[-1,1]$. Put $w=\frac{4|z_n|^2}{(1+|z|^2)^2+t^2}$. For $(z,t)=(0,z_n,0)$ with $|z_n|=1$ we have $w=1$ and hence
$$ u(z,t) = 2^{-\frac{a+2n}{2}}F(1). $$
On the other hand, using the integral expression \eqref{eq:GenIntRepPoissonTrafo} we find
\begin{align*}
 P_af(0,z_n,0) &= c_{n,a} \int_{H^{2n-1}} \frac{1}{|(z',-z_n,t')|^{2n-a}}((1+|z'|^2)^2+t'^2)^{-\frac{a+2n}{4}} \,d(z',t')\\
 &= 2\vol(S^{2n-3})c_{n,a}\int_0^\infty\int_0^\infty ((1+r^2)^2+s^2)^{-n} r^{2n-3} \,dr\,ds\\
 &= \frac{2\pi^{n-1}B(\frac{1}{2},n-\frac{1}{2})}{\Gamma(n-1)}c_{n,a} \int_0^\infty (1+r^2)^{1-2n}r^{2n-3} \,dr\\
 &= \frac{\pi^{n-1}B(\frac{1}{2},n-\frac{1}{2})B(n-1,n)}{\Gamma(n-1)}c_{n,a},
\end{align*}
where we have used the integral formula \cite[equation 3.251~(11)]{GR07} twice. This implies
$$ F(1) = \frac{2^{\frac{a+2n}{2}}\pi^{n-\frac{1}{2}}\Gamma(n-\frac{1}{2})}{\Gamma(2n-1)}c_{n,a}. $$
Now, by Corollary~\ref{cor:SolutionsOnDistributions} we know that $\calL_au=0$. An elementary calculation shows that
\begin{multline*}
 \calL_au(z,t) = 4((1+|z|^2)^2+t^2)^{-\frac{a+2n}{4}}\cdot w\Big(w(1-w)F''(w)+(\tfrac{a+2}{2}-\tfrac{a+2n+2}{2}w)F'(w)-(\tfrac{a+2n}{4})^2F(w)\Big).
\end{multline*}
The differential equation for $F$ is of hypergeometric type. At the regular singularity $w=1$ it has exponents $0$ and $1-n\in-\NN$ and hence there exist two linear indepentent solutions $F_1$ and $F_2$ with asymptotic behaviour $F_1(w)\sim 1$ and $F_2(w)\sim(1-w)^{1-n}$ as $w\to1$.  Since $F(w)$ is regular at $w=1$ we must have
\begin{align*}
 F(w) &= \frac{2^{\frac{a+2n}{2}}\pi^{n-\frac{1}{2}}\Gamma(n-\frac{1}{2})}{\Gamma(2n-1)}c_{n,a}\cdot F_1(w)\\
 &= \frac{2^{\frac{a+2n}{2}}\pi^{n-\frac{1}{2}}\Gamma(n-\frac{1}{2})}{\Gamma(2n-1)}c_{n,a}\cdot{_2F_1}\left(\frac{a+2n}{4},\frac{a+2n}{4};n;1-w\right)
\end{align*}
which shows the claim.
\end{proof}

We can now compute the constant $c_{n,a}$ by putting $z_n=0$ in the above formula for $P_af(z,t)$. By \cite[Theorem 2.2.2]{AAR99} and the duplication formula for the gamma function we have
$$ P_af(z',0,t') = \frac{2^{\frac{a-2n+4}{2}}\pi^n\Gamma(-\frac{a}{2})}{\Gamma(\frac{2n-a}{4})^2}c_{n,a}\cdot f(z',t'). $$
Since $(P_af)|_{H^{2n-1}}=f$ we must have
$$ c_{n,a} = \frac{2^{\frac{2n-a-4}{2}}\Gamma(\frac{2n-a}{4})^2}{\pi^n\Gamma(-\frac{a}{2})}. $$

\subsubsection{Constant boundary values}

Let $\1_{H^{2k+1}}$ denote the constant function with value $1$ on the Heisenberg group $H^{2k+1}$.

\begin{lemma}\label{lem:ExplicitBVConstantCplx}
Let $n\geq2$. For $a<0$ we have
$$ P_a\1_{H^{2n-1}} = \1_{H^{2n+1}}. $$
\end{lemma}

\begin{proof}
We have
\begin{align*}
 P_a\1_{H^{2n-1}}(z,t) &= c_{n,a} \int_{H^{2n-1}} \frac{|z_n|^{-a}}{|(z,t)^{-1}\cdot(z',0,t')|^{2n-a}}\,d(z',t')\\
 &= c_{n,a}|z_n|^{-a} \int_{\CC^{n-1}}\int_\RR |(z',z_n,t')|^{a-2n}\,dt'\,dz'\\
 &= c_{n,a}|z_n|^{-a} \left(\int_{\CC^{n-1}} (|z'|^2+|z_n|^2)^{\frac{a-2n+2}{2}}\,dz'\right)\left(\int_\RR(1+t'^2)^{\frac{a-2n}{4}}\,dt'\right)\\
 &= \frac{4c_{n,a}\pi^{n-1}}{\Gamma(n-1)} \left(\int_0^\infty (1+r^2)^{\frac{a-2n+2}{2}}r^{2n-3}\,dr\right)\left(\int_0^\infty(1+s^2)^{\frac{a-2n}{4}}\,ds\right).
\end{align*}
Evaluating the two integrals using \cite[formula 3.251~(2)]{GR07} shows the claim.
\end{proof}

\section{Computations in the real case}\label{sec:RealCase}

\subsection{The Casimir operator}\label{sec:CasimirReal}

An explicit basis of the Lie algebra $\frakg=\frako(1,n+1)$ is given by the generator $H$ of $\fraka$ and the elements
\begin{align*}
 M_{jk} &:= E_{j+2,k+2}-E_{k+2,j+2}, && 1\leq j<k\leq n,\\
 X_j &:= E_{j+2,1}-E_{j+2,2}+E_{1,j+2}+E_{2,j+2}, && 1\leq j\leq n,\\
 \overline{X}_j &:= E_{j+2,1}+E_{j+2,2}+E_{1,j+2}-E_{2,j+2}, && 1\leq j\leq n.
\end{align*}
Here $M_{jk}$ span $\frakm$, $X_j$ span $\frakn$ and $\overline{X}_j$ span $\overline{\frakn}$.

The action of the generators $MA\overline{N}$ and $w_0$ of $G$ on $I_\mu^\infty$ is then given by
\begin{align*}
 \pi_\mu^\infty(\overline{n}_{x'})f(x) &= f(x-x'), && \overline{n}_x\in\overline{N},\\
 \pi_\mu^\infty(\diag(\lambda,\lambda,m))f(x) &= f(\lambda m^{-1}x), && \lambda\in O(1),m\in O(n),\\
 \pi_\mu^\infty(e^{sH})f(x) &= e^{(\mu+\rho)s}f(e^sx), && e^{sH}\in A,\\
 \pi_\mu^\infty(w_0)f(x) &= |x|^{-2(\mu+\rho)}f\left(-\frac{x}{|x|^2}\right).
\end{align*}
This immediately gives the action of the differential representation on $\frakm$, $\fraka$ and $\overline{\frakn}$:
$$ d\pi_\mu^\infty(M_{jk}) = x_j\frac{\partial}{\partial x_k}-x_k\frac{\partial}{\partial x_j}, \qquad d\pi_\mu^\infty(H) = E+\mu+\rho, \qquad d\pi_\mu^\infty(\overline{X}_j) = -\frac{\partial}{\partial x_j}, $$
where $E=\sum_{k=1}^nx_k\frac{\partial}{\partial x_k}$ denotes the Euler operator on $\RR^n$. Thanks to the relation $\Ad(w_0)\overline{X}_j=-X_j$ we have
$$ d\pi_\mu^\infty(X_j) = -\pi_\mu^\infty(w_0)d\pi_\mu^\infty(\overline{X}_j)\pi_\mu^\infty(w_0), $$
which is easily shown to be equal to
$$ d\pi_\mu^\infty(X_j) = -|x|^2\frac{\partial}{\partial x_j}+2x_j(E+\mu+\rho). $$

Now the Casimir element $C=C_1$ can be expressed using the above constructed basis of $\frakg$:
$$ C = H^2-(n-1)H-\sum_{1\leq j<k\leq n-1}M_{jk}^2+\sum_{j=1}^{n-1}X_j\overline{X}_j. $$
An elementary calculation using the previously derived formulas for the differential representation shows that
\begin{equation}
 d\pi_\mu^\infty(C) = x_n^2\Delta+2(\mu+1)x_n\frac{\partial}{\partial x_n}+(\mu+\rho)(\mu-\rho+1),\label{eq:FormulaCasimirReal}
\end{equation}
where $\Delta=\sum_{k=1}^n\frac{\partial^2}{\partial x_k^2}$ denotes the Laplacian on $\RR^n$.

\subsection{Uniqueness}\label{sec:UniquenessReal}

For the proof of uniqueness of solutions to the Dirichlet problem and for the isometry property of the Poisson transform $P_a$ we use Euclidean Fourier analysis (see Section~\ref{sec:SobolevSpacesReal} for the notation). Consider the Euclidean Fourier transform
$$ \calF_{\RR^n}:\dot{H}^{\frac{2-a}{2}}(\RR^n)\to L^2(\RR^n,|\xi|^{2-a}d\xi), \quad u\mapsto\widehat{u}. $$
Then we have
$$ \widehat{\Delta_au}(\xi) = \Big(\partial_{\xi_n}^2|\xi|^2-a\partial_{\xi_n}\xi_n\Big)\widehat{u}(\xi) = \Big(|\xi|^2\partial_{\xi_n}^2-(a-4)\xi_n\partial_{\xi_n}+(2-a)\Big)\widehat{u}(\xi) $$
and
$$ \widehat{u|_{\RR^{n-1}}}(\xi') = \frac{1}{\sqrt{2\pi}} \int_{-\infty}^\infty \widehat{u}(\xi',\xi_n) \,d\xi_n. $$

Let $z=\frac{\xi_n}{|\xi'|}$ and define
\begin{equation}
 v(\xi',z) := \widehat{u}(\xi',|\xi'|z).\label{eq:DefinitionFctV}
\end{equation}
Then
$$ \widehat{\Delta_au}(\xi) = \calD_{a,z}v(\xi',z) \qquad \text{and} \qquad \widehat{u|_{\RR^{n-1}}}(\xi') = \frac{|\xi'|}{\sqrt{2\pi}} \int_{-\infty}^\infty v(\xi',z) \,dz, $$
where $\calD_{a,z}$ is the ordinary differential operator
$$ \calD_{a,z} := (1+z^2)\frac{d^2}{dz^2}-(a-4)z\frac{d}{dz}-(a-2). $$

Now let $u\in\dot{H}^{\frac{2-a}{2}}(\RR^n)$ be a solution of the Dirichlet problem \eqref{eq:RealBdyValueProblem} with boundary value $f\in\dot{H}^{\frac{1-a}{2}}(\RR^{n-1})$. Then the corresponding function $v\in L^2(\RR^n,|\xi|^{2-a}d\xi)$ defined in \eqref{eq:DefinitionFctV} satisfies
$$ \calD_{a,z}v(\xi',z)=0 \qquad \text{and} \qquad \int_{-\infty}^\infty v(\xi',z) \,dz = \frac{\sqrt{2\pi}}{|\xi'|}\widehat{f}(\xi') $$
for almost every $\xi\in\RR^{n-1}$. This implies
$$ v(\xi',z) = \frac{\sqrt{2\pi}}{|\xi'|}\widehat{f}(\xi')\phi(z) $$
where $\phi(z)$ satisfies
$$ \calD_{a,z}\phi=0 \qquad \text{and} \qquad \int_{-\infty}^\infty \phi(z) \,dz = 1. $$
The equation $\calD_{a,z}\phi=0$ has a fundamental system of solutions spanned by
\begin{align*}
 \phi_1(z) &= {_2F_1}\left(\frac{1}{2},\frac{2-a}{2};\frac{1}{2};-z^2\right) = (1+z^2)^{\frac{a-2}{2}},\\
 \phi_2(z) &= z\cdot{_2F_1}\left(1,\frac{3-a}{2};\frac{3}{2};-z^2\right).
\end{align*}
We note that the property $\widehat{u}\in L^2(\RR^n,|\xi|^{2-a}d\xi)$ is equivalent to $\widehat{f}\in L^2(\RR^n,|\xi'|^{1-a}d\xi')$ and $\phi\in L^2(\RR,(1+z^2)^{\frac{2-a}{2}}dz)$. By the asymptotics of the hypergeometric function (see e.g. \cite[Theorem 2.3.2]{AAR99}) $\phi_2$ is not contained in $L^2(\RR,(1+z^2)^{\frac{2-a}{2}}dz)$ and hence
\begin{equation}
 v(\xi',z) = \frac{\sqrt{2}\Gamma(\frac{2-a}{2})}{\Gamma(\frac{1-a}{2})}|\xi'|^{-1}\widehat{f}(\xi')(1+z^2)^{\frac{a-2}{2}},\label{eq:SolutionRealFTside1}
\end{equation}
where we used the following integral formula (see e.g. \cite[equation~3.251~(2)]{GR07}):
$$ \int_{-\infty}^\infty\phi_1(z) \,dz=\frac{\sqrt{\pi}\Gamma(\frac{1-a}{2})}{\Gamma(\frac{2-a}{2})}. $$
This determines $v$ and hence $u$ uniquely, proving uniqueness of solutions of \eqref{eq:RealBdyValueProblem}.


\subsection{Isometry}\label{sec:IsometryReal}

Let $f\in\dot{H}^{\frac{1-a}{2}}(\RR^{n-1})$ and $u=P_af\in\dot{H}^{\frac{2-a}{2}}(\RR^n)$ then, using \eqref{eq:IsometryFourierTrafoSobolevL2}, \eqref{eq:DefinitionFctV} and \eqref{eq:SolutionRealFTside1} we find
\begin{align*}
 \|u\|_{\dot{H}^{\frac{2-a}{2}}(\RR^n)}^2 &= \|\widehat{u}\|_{L^2(\RR^n,|\xi|^{2-a}d\xi)}^2 = \int_{\RR^n} |\widehat{u}(\xi)|^2|\xi|^{2-a} \,d\xi\\
 &= \int_{\RR^{n-1}}\int_{-\infty}^\infty |\xi'|\cdot|v(\xi',z)|^2 (|\xi'|^2+|\xi'|^2z^2)^{\frac{2-a}{2}} \,dz \,d\xi'\\
 &= \frac{2\Gamma(\frac{2-a}{2})^2}{\Gamma(\frac{1-a}{2})^2} \left(\int_{-\infty}^\infty(1+z^2)^{\frac{a-2}{2}}\,dz\right) \int_{\RR^{n-1}} |\widehat{f}(\xi')|^2 |\xi'|^{1-a}\,d\xi'\\
 &= \frac{2\sqrt{\pi}\Gamma(\frac{2-a}{2})}{\Gamma(\frac{1-a}{2})} \|f\|_{\dot{H}^{\frac{1-a}{2}}(\RR^{n-1})}^2.
\end{align*}
In view of the duplication formula for the Gamma function this shows \eqref{eq:IsometryReal1}. 

\subsection{$L^p$-$L^q$ boundedness}\label{sec:LpReal}

We now show Theorem~\ref{thm:IsometryReal}~(2). Parts of the proof can also be found in \cite{Che14}. We include a full proof for completeness.

\begin{proposition}
For any $1<p\leq\infty$ and $q=\frac{n}{n-1}p$ the Poisson transform $P_a$ is a bounded operator
$$ P_a: L^p(\RR^{n-1})\to L^q(\RR^n). $$
More precisely,
$$ \Vert P_af\Vert_q \leq (2c_{n,a}^{\frac{1}{n-1}})^{\frac{1}{q}}\Vert f\Vert_p, \qquad \mbox{for all $f\in L^p(\RR^{n-1})$.} $$
\end{proposition}

\begin{proof}
By the Marcinkiewicz Interpolation Theorem it suffices to show that $P_a$ is a bounded operator
$$ L^\infty(\RR^{n-1})\to L^\infty(\RR^n) \qquad \mbox{and} \qquad L^1(\RR^{n-1})\to L^{\frac{n}{n-1}}_w(\RR^n), $$
where $L^r_w(\RR^n)$ stands for the weak type $L^r$-space. First observe that $P_a\1_{\RR^{n-1}}=\1_{\RR^n}$ by Lemma~\ref{lem:ExplicitBVConstantReal} and that the integral kernel of $P_a$ is a positive function. Thus we have
$$ P_a: L^\infty(\RR^{n-1})\to L^\infty(\RR^n), \qquad \Vert P_af\Vert_{\infty}\le\Vert f\Vert_{\infty}. $$
We now prove the weak type inequality
$$ P_a: L^1(\RR^{n-1})\to L^{\frac{n}{n-1}}_w(\RR^n), \qquad \Vert P_af\Vert_{\frac{n}{n-1},w}\leq (2^{n-1}c_{n,a})^{\frac{1}{n}}\Vert f\Vert_1. $$
First note that
$$ |P_af(z,t)| \leq c_{n,a} \int_{\RR^{n-1}} \frac{|x_n|^{1-a}}{(|x'-y|^2+x_n^2)^{\frac{n-a}{2}}} |f(y)|\,dy \leq c_{n,a} |x_n|^{1-n} \Vert f\Vert_1. $$
Hence $|P_af(z,t)|>\lambda$ implies $|x_n|<(\frac{c_{n,a}\Vert f\Vert_1}{\lambda})^{\frac{1}{n-1}}$. Denote by $m$ the Lebesgue measure on $\RR^n$. Let $b=(\frac{c_{n,a}\Vert f\Vert_1}{\lambda})^{\frac{1}{n-1}}$, then Chebyshev’s Inequality yields
\begin{align*}
 & m(\{x\in\RR^n:|P_af(x)|>\lambda\}) = m(\{x\in\RR^n:|x_n|<b,\,|P_af(x)|>\lambda\})\\
 \leq{}& \frac{1}{\lambda} \int_{x\in\RR^n,\,|x_n|<b} |P_af(x)|\,dx\\
 \leq{}& \frac{c_{n,a}}{\lambda} \int_{y\in\RR^{n-1}} |f(y)| \int_{x\in\RR^n,\,|x_n|<b} \frac{|x_n|^{1-a}}{(|x'-y|^2+x_n^2)^{\frac{n-a}{2}}} \,dx\,dy\\
 ={}& \frac{1}{\lambda}\int_{\RR^{n-1}} |f(y)| \int_{|x_n|<b} P_a\1(y,x_n)\,dx_n\,dy\\
 ={}& \frac{1}{\lambda}\Vert f\Vert_1\int_{|x_n|<b}dx_n = 2c_{n,a}^{\frac{1}{n-1}}\left(\frac{\Vert f\Vert_1}{\lambda}\right)^{\frac{n}{n-1}},
\end{align*}
where we have used $P_a\1=\1$ in the fifth step. Hence
$$ \Vert P_af\Vert_{\frac{n}{n-1},w} \leq (2^{n-1}c_{n,a})^{\frac{1}{n}}\Vert f\Vert_1 $$
and the claim follows.
\end{proof}

\section{Computations in the the complex case}\label{sec:CplxCase}

\subsection{The Casimir operator}\label{sec:CasimirCplx}

An explicit basis of the Lie algebra $\frakg=\fraku(1,n+1)$ is given by the generator $H$ of $\fraka$ and the elements
\begin{align*}
 M_0 &:= i(E_{1,1}+E_{2,2}),\\
 M_j &:= iE_{j,j}, && 1\leq j\leq n,\\
 M_{jk}^{(1)} &:= E_{j+2,k+2}-E_{k+2,j+2}, && 1\leq j<k\leq n,\\
 M_{jk}^{(2)} &:= i(E_{j+2,k+2}+E_{k+2,j+2}), && 1\leq j<k\leq n,\\
 X_j &:= E_{j+2,1}-E_{j+2,2}+E_{1,j+2}+E_{2,j+2}, && 1\leq j\leq n,\\
 Y_j &:= i(E_{j+2,1}-E_{j+2,2}-E_{1,j+2}-E_{2,j+2}), && 1\leq j\leq n,\\
 T &:= -\tfrac{1}{2}i(E_{1,1}-E_{1,2}+E_{2,1}-E_{2,2}),\\
 \overline{X}_j &:= E_{j+2,1}+E_{j+2,2}+E_{1,j+2}-E_{2,j+2}, && 1\leq j\leq n,\\
 \overline{Y}_j &:= i(E_{j+2,1}+E_{j+2,2}-E_{1,j+2}+E_{2,j+2}), && 1\leq j\leq n,\\
 \overline{T} &:= -\tfrac{1}{2}i(E_{1,1}+E_{1,2}-E_{2,1}-E_{2,2}).
\end{align*}

The action of the generators $MA\overline{N}$ and $w_0$ of $G$ on $I_\mu^\infty$ is then given by
\begin{align*}
 \pi_\mu^\infty(\overline{n}_{z',t'})f(z,t) &= f((-z',-t')\cdot(z,t)), && \overline{n}_{z',t'}\in\overline{N},\\
 \pi_\mu^\infty(\diag(\lambda,\lambda,m))f(z,t) &= f(\lambda m^{-1}z,t), && \lambda\in U(1),m\in U(n),\\
 \pi_\mu^\infty(e^{sH})f(z,t) &= e^{(\mu+\rho)s}f(e^sz,e^{2s}t), && e^{sH}\in A,\\
 \pi_\mu^\infty(w_0)f(z,t) &= |(z,t)|^{-2(\mu+\rho)}f\left(-\frac{z}{|z|^2+it},-\frac{t}{|z|^4+t^2}\right).
\end{align*}
This immediately gives the action of the differential representation on $\frakm$, $\fraka$ and $\overline{\frakn}$:
\begin{align*}
 d\pi_\mu^\infty(M_0) &= \sum_{j=1}^n\left(x_j\frac{\partial}{\partial y_j}-y_j\frac{\partial}{\partial x_j}\right),\\
 d\pi_\mu^\infty(M_j) &= y_j\frac{\partial}{\partial x_j}-x_j\frac{\partial}{\partial y_j},\\
 d\pi_\mu^\infty(M_{jk}^{(1)}) &= \left(x_j\frac{\partial}{\partial x_k}-x_k\frac{\partial}{\partial x_j}\right)+\left(y_j\frac{\partial}{\partial y_k}-y_k\frac{\partial}{\partial y_j}\right),\\
 d\pi_\mu^\infty(M_{jk}^{(2)}) &= \left(y_j\frac{\partial}{\partial x_k}-x_k\frac{\partial}{\partial y_j}\right)-\left(x_j\frac{\partial}{\partial y_k}-y_k\frac{\partial}{\partial x_j}\right),\\
 d\pi_\mu^\infty(H) &= \sum_{j=1}^n\Big(x_j\frac{\partial}{\partial x_j}+y_j\frac{\partial}{\partial y_j}\Big)+2t\frac{\partial}{\partial t}+(\mu+\rho),\\
 d\pi_\mu^\infty(\overline{X}_j) &= -\frac{\partial}{\partial x_j}+2y_j\frac{\partial}{\partial t},\\
 d\pi_\mu^\infty(\overline{Y}_j) &= -\frac{\partial}{\partial y_j}-2x_j\frac{\partial}{\partial t},\\
 d\pi_\mu^\infty(\overline{T}) &= -\frac{\partial}{\partial t}.
\end{align*}
Thanks to the relations $\Ad(w_0)\overline{X}_j=-X_j$, $\Ad(w_0)\overline{Y}_j=-Y_j$ and $\Ad(w_0)\overline{T}=T$ we have
\begin{align*}
 d\pi_\mu^\infty(X_j) &= -\pi_\mu(w_0)d\pi_\mu(\overline{X}_j)\pi_\mu(w_0),\\
 d\pi_\mu^\infty(Y_j) &= -\pi_\mu(w_0)d\pi_\mu(\overline{Y}_j)\pi_\mu(w_0),\\
 d\pi_\mu^\infty(T) &= \pi_\mu(w_0)d\pi_\mu(\overline{T})\pi_\mu(w_0),
\end{align*}
which, after an easy though longish calculation, turn out to be equal to
\begin{align*}
 d\pi_\mu^\infty(X_j) ={}& 2\sum_{k=1}^n\Big((x_jx_k-y_jy_k)\frac{\partial}{\partial x_k}+(x_jy_k+y_jx_k)\frac{\partial}{\partial y_k}\Big)\\
 &\hspace{1cm}-\Big(|z|^2\frac{\partial}{\partial x_j}-t\frac{\partial}{\partial y_j}\Big)+2(tx_j-|z|^2y_j)\frac{\partial}{\partial t}+2(\mu+\rho)x_j,\\
 d\pi_\mu^\infty(Y_j) ={}& 2\sum_{k=1}^n\Big((x_jy_k+y_jx_k)\frac{\partial}{\partial x_k}-(x_jx_k-y_jy_k)\frac{\partial}{\partial y_k}\Big)\\
 &\hspace{1cm}-\Big(t\frac{\partial}{\partial x_j}+|z|^2\frac{\partial}{\partial y_j}\Big)+2(|z|^2x_j+ty_j)\frac{\partial}{\partial t}+2(\mu+\rho)y_j,\\
 d\pi_\mu^\infty(T) ={}& -\sum_{k=1}^n\Big((tx_k-|z|^2y_k)\frac{\partial}{\partial x_k}+(|z|^2x_k+ty_k)\frac{\partial}{\partial y_k}\Big)\\
 &\hspace{5.5cm}+(|z|^4-t^2)\frac{\partial}{\partial t}-(\mu+\rho)t.
\end{align*}

Now the Casimir elements $C_1$ and $C_2$ can be expressed using the above constructed basis of $\frakg$:
\begin{align*}
 C_1 ={}& H^2-2nH-\sum_{1\leq j<k\leq n-1}\Big((M_{jk}^{(1)})^2+(M_{jk}^{(2)})^2\Big)-2\sum_{j=1}^{n-1}M_j^2\\
 &\hspace{4cm}-M_0^2+\sum_{j=1}^{n-1}\Big(X_j\overline{X}_j+Y_j\overline{Y}_j\Big)-4T\overline{T},\\
 C_2 ={}& -2M_n^2.
\end{align*}
An elementary calculation using the previously derived formulas for the differential representation shows that $C=C_1-C_2$ acts by
\begin{equation}
 d\pi_\mu^\infty(C) = |z_n|^2\calL+2(\mu+1)\left(x_n\frac{\partial}{\partial x_n}+y_n\frac{\partial}{\partial y_n}\right)+(\mu+\rho)(\mu-\rho+2),\label{eq:FormulaCasimirComplex}
\end{equation}
where
\begin{equation}
 \calL = \sum_{j=1}^n\left(\left(\frac{\partial}{\partial x_j}+2y_j\frac{\partial}{\partial t}\right)^2+\left(\frac{\partial}{\partial y_j}-2x_j\frac{\partial}{\partial t}\right)^2\right)\label{eq:CRLaplacian}
\end{equation}
denotes the left-invariant CR-Laplacian on the Heisenberg group $\overline{N}$.

\subsection{Uniqueness}\label{sec:UniquenessCplx}

For the proof of uniqueness of $z_n$-radial solutions to the Dirichlet problem and for the isometry property of the Poisson transform $P_a$ we use Fourier analysis on the Heisenberg group $H^{2n+1}$ (see Section~\ref{sec:SobolevSpacesCplx} for the notation).

For $\mu\in\RR^\times$ consider the representations $\sigma_\mu$ on the Fock space
$$ \calF_\mu := \left\{\xi\in\calO(\CC^n):\|\xi\|_\mu^2=\int_{\CC^n}|\xi(w)|^2e^{-2|\mu||w|^2}\,dw<\infty\right\}. $$
Splitting variables $w=(w',w_n)$ with $w'=(w_1,\ldots,w_{n-1})\in\CC^{n-1}$ and $w_n\in\CC$ the Fock space $\calF_\mu$ can be written as the Hilbert space tensor product $\calF_\mu'\otimes\calF_\mu''$ where $\calF_\mu'$ and $\calF_\mu''$ are the corresponding Fock spaces on $\CC^{n-1}$ and $\CC$, respectively. Elementary tensors are functions of the form $\xi(w')\eta(w_n)$ with $\xi\in\calF_\mu'$ and $\eta\in\calF_\mu''$. We further let $\calP'$ resp. $\calP''$ be the space of polynomials on $\CC^{n-1}$ resp. $\CC$ and $\calP_k'$ resp. $\calP_k''$ its subspace of homogeneous polynomials of degree $k$.

Now let $u\in\dot{H}^{\frac{2-a}{2}}(H^{2n+1})$ be a solution of the Dirichlet problem
\begin{equation*}
 \calL_a u = 0, \qquad u|_{H^{2n-1}} = 0.\label{eq:HomDirichletProblemCplz}
\end{equation*}
Write $(z,t)=(z',z_n,t)$ with $z'=(z_1,\ldots,z_{n-1})\in\CC^{n-1}$ and assume that $u$ is $z_n$-radial, i.e.
$$ u_m(z',e^{i\theta}z_n,t)=u_m(z',z_n,t). $$
In this case it is easy to see that $\sigma_\mu(u)$ maps $\calF_\mu'\otimes\calP_\ell''$ into $\calF_\mu'\otimes\calP_\ell''$. Fixing $\mu\in\RR^\times$ we can therefore write
$$ \sigma_\mu(u)|_{\calP_k'\otimes\calP_\ell''} = T_{\mu,k,\ell}\otimes\id_{\calP_\ell''} $$
with operators $T_{\mu,k,\ell}:\calP_k'\to\calF_\mu'$. For convenience we also put $T_{\mu,k,-1}:=0$.

We first study how the differential equation $\calL_au=0$ is expressed in terms of $T_{\mu,k,\ell}$.

\begin{lemma}
If $\calL_a u=0$ then for all $\mu\in\RR^\times$ and $k,\ell\in\NN$ we have
\begin{multline}
 (\ell+1)(2(2k+2\ell+n+2)-a)T_{\mu,k,\ell+1}-(2(2\ell+1)(2k+2\ell+n)+a)T_{\mu,k,\ell}\\
 +\ell(2(2k+2\ell+n-2)+a)T_{\mu,k,\ell-1}=0.\label{eq:PDEinCplxFourierPicture}
\end{multline}
\end{lemma}

\begin{proof}
We have
$$ \sigma_\mu(\calL u) = -4|\mu|\sigma_\mu(u)\circ(2E+n), $$
where $E=\sum_{j=1}^nw_j\frac{\partial}{\partial w_j}$ is the Euler operator acting on $\calP_k$ by the scalar $k$. We further calculate
$$ \sigma_\mu(|z_n|^2u) = \frac{1}{2|\mu|}\left(\sigma_\mu(u)+\sigma_\mu(u)w_n\frac{\partial}{\partial w_n}+w_n\frac{\partial}{\partial w_n}\sigma_\mu(u)-w_n\sigma_\mu(u)\frac{\partial}{\partial w_n}-\frac{\partial}{\partial w_n}\sigma_\mu(u)w_n\right) $$
and
$$ \sigma_\mu\left(\left(x_n\frac{\partial}{\partial x_n}+y_n\frac{\partial}{\partial y_n}\right)u\right) = w_n\sigma_\mu(u)\frac{\partial}{\partial w_n}-\frac{\partial}{\partial w_n}\sigma_\mu(u)w_n-\sigma_\mu(u). $$
This gives
\begin{multline*}
 \sigma_\mu(\calL_au) = -2\Bigg(\sigma_\mu(u)(2E+n)+\sigma_\mu(u)(2E+n)w_n\frac{\partial}{\partial w_n}+w_n\frac{\partial}{\partial w_n}\sigma_\mu(u)(2E+n)\\
 -w_n\sigma_\mu(u)(2E+n)\frac{\partial}{\partial w_n}-\frac{\partial}{\partial w_n}\sigma_\mu(u)(2E+n)w_n\Bigg)\\
 +a\Bigg(w_n\sigma_\mu(u)\frac{\partial}{\partial w_n}-\frac{\partial}{\partial w_n}\sigma_\mu(u)w_n-\sigma_\mu(u)\Bigg).
\end{multline*}
Using $w_n\frac{\partial}{\partial w_n}|_{\calP_\ell''}=\ell$ and $E|_{\calP_k'\otimes\calP_\ell''}=(k+\ell)$ we find
\begin{multline*}
 \sigma_\mu(\calL_au)|_{\calP_k'\otimes\calP_\ell''} = \Bigg((\ell+1)(2(2k+2\ell+n+2)-a)T_{\mu,k,\ell+1}\\
 -(2(2\ell+1)(2k+2\ell+n)+a)T_{\mu,k,\ell}\\
 +\ell(2(2k+2\ell+n-2)+a)T_{\mu,k,\ell-1}\Bigg)\otimes\id_{\calP_\ell''}
\end{multline*}
which implies the claim.
\end{proof}

Next we study what the boundary condition $u|_{H^{2n-1}}=0$ implies for the operators $T_{\mu,k,\ell}$.

\begin{lemma}\label{lem:RestrictionZeroFourierSideCplx}
If $u|_{H^{2n-1}}=0$ then for all $\mu\in\RR^\times$ and $k\in\NN$ we have
$$ \sum_{\ell=0}^\infty T_{\mu,k,\ell} = 0. $$
\end{lemma}

\begin{proof}
In \cite[Lemma 4.4]{MOZ14} we showed that the boundary value map $\dot{H}^{s}(H^{2n+1})\to\dot{H}^{s-1}(H^{2n-1})$ is in the Fourier transformed picture given by
$$ \sigma_\mu'(u|_{H^{2n-1}}) = \frac{2|\mu|}{\pi} \ptr(\sigma_\mu(u)), $$
where $\ptr(\sigma_\mu(u))$ is the partial trace of the operator on $\calF_\mu$ with respect to the second factor in the decomposition $\calF_\mu=\calF_\mu'\otimes\calF_\mu''$, i.e.
$$ \ptr(\sigma_\mu(u))|_{\calP_k'} = \sum_{\ell=0}^\infty T_{\mu,k,\ell}. $$
Hence $u|_{H^{2n-1}}=0$ implies that for every $k\in\NN$ we have
\begin{equation*}
 0 = \ptr(\sigma_\mu(u))|_{\calP_k'} = \sum_{\ell=0}^\infty T_{\mu,k,\ell}.\qedhere
\end{equation*}
\end{proof}

Now, since $a<2n+4$ the coefficient of $T_{\mu,k,\ell+1}$ in \eqref{eq:PDEinCplxFourierPicture} never vanishes and $T_{\mu,k,\ell}$ is uniquely determined by $T_{\mu,k,0}$ by the formula
\begin{equation}
 T_{\mu,k,\ell} = \frac{(k+\frac{n}{2}+\frac{a}{4})_\ell}{(k+\frac{n}{2}-\frac{a}{4}+1)_\ell}T_{\mu,k,0}.\label{eq:FormulaTkl}
\end{equation}

Using \eqref{eq:FormulaTkl} this yields
$$ \sum_{\ell=0}^\infty T_{\mu,k,\ell} = \left(\sum_{\ell=0}^\infty \frac{(k+\frac{n}{2}+\frac{a}{4})_\ell}{(k+\frac{n}{2}-\frac{a}{4}+1)_\ell}\right)T_{\mu,k,0} = \frac{a-2n-4k}{2a}T_{\mu,k,0}, $$
where we have used the following identity which holds for $\Re(y-x)>1$:
\begin{equation}
 \sum_{m=0}^\infty\frac{(x)_m}{(y)_m} = {_2F_1}(1,x;y;1) = \frac{\Gamma(y)\Gamma(y-x-1)}{\Gamma(y-x)\Gamma(y-1)} = \frac{y-1}{y-x-1}.\label{eq:IdentitySumPochhammerQuotient}
\end{equation}

Now $u|_{H^{2n-1}}=0$ implies by Lemma~\ref{lem:RestrictionZeroFourierSideCplx} and the previous calculation that $T_{\mu,k,0}=0$ for all $\mu\in\RR^\times$ and $k\in\NN$ (note that $-2n<a<0$ and thus the coefficient $(a-2n-4k)$ is non-zero). In view of formula \eqref{eq:FormulaTkl} this yields $T_{\mu,k,\ell}=0$ for all $k,\ell\in\NN$ and hence $\sigma_\mu(u)=0$ for all $\mu\in\RR^\times$. Therefore $u=0$ by the Fourier inversion formula which shows uniqueness.

\subsection{Isometry}\label{sec:IsometryCplx}

Let $f\in\dot{H}^{-\frac{a}{2}}(H^{2n-1})$ and let $T_{\mu,k}=\sigma_\mu(f)|_{\calP_k}$. In the last section we have seen that $u=P_af$ is given by
$$ \sigma_\mu(u)|_{\calP_k'\otimes\calP_\ell''}=T_{\mu,k,\ell}\otimes\id_{\calP_\ell''} $$
with
$$ T_{\mu,k,\ell} = \frac{(k+\frac{n}{2}+\frac{a}{4})_\ell}{(k+\frac{n}{2}-\frac{a}{4}+1)_\ell}T_{\mu,k,0} = \frac{(k+\frac{n}{2}+\frac{a}{4})_\ell}{(k+\frac{n}{2}-\frac{a}{4}+1)_\ell}\frac{a}{a-2n-4k}\frac{\pi}{|\mu|}T_{\mu,k}. $$
This means that for $m\in\NN$ we have
$$ \|P_m\circ\sigma_\mu(u)\|_{\HS(\calF_\mu)}^2 = \sum_{k=0}^m \|T_{\mu,k,m-k}\|_{\HS(\calF_\mu')}^2 = \frac{\pi^2a^2}{|\mu|^2}\sum_{k=0}^m \frac{(k+\frac{n}{2}+\frac{a}{4})_{m-k}^2}{(4k+2n-a)^2(k+\frac{n}{2}-\frac{a}{4}+1)_{m-k}^2}\|T_{\mu,k}\|_{\HS(\calF_\mu')}^2. $$

Using \eqref{eq:SobolevNormCplx} we calculate
\begin{align*}
 \|u\|_{\frac{2-a}{2}}^2 &= \frac{2^{n-1}a^2}{\pi^{n-1}} \sum_{m=0}^\infty\sum_{k=0}^m \frac{(1+\frac{n}{2}-\frac{a}{4})_m(k+\frac{n}{2}+\frac{a}{4})_{m-k}^2}{(\frac{n}{2}+\frac{a}{4})_m(4k+2n-a)^2(k+\frac{n}{2}-\frac{a}{4}+1)_{m-k}^2} \int_\RR \|T_{\mu,k}\|_{\HS(\calF_\mu')}^2 |\mu|^{n+s-2} \,d\mu\\
 &= \frac{2^{n-1}a^2}{\pi^{n-1}} \sum_{k=0}^\infty \frac{1}{(4k+2n-a)^2} \sum_{m=k}^\infty \frac{(1+\frac{n}{2}-\frac{a}{4})_m(k+\frac{n}{2}+\frac{a}{4})_{m-k}^2}{(\frac{n}{2}+\frac{a}{4})_m(k+\frac{n}{2}-\frac{a}{4}+1)_{m-k}^2} \int_\RR \|T_{\mu,k}\|_{\HS(\calF_\mu')}^2 |\mu|^{n+s-2} \,d\mu\\
 &= \frac{2^{n-1}a^2}{\pi^{n-1}} \sum_{k=0}^\infty \frac{(1+\frac{n}{2}-\frac{a}{4})_k}{(4k+2n-a)^2(\frac{n}{2}+\frac{a}{4})_k} \sum_{m=0}^\infty \frac{(k+\frac{n}{2}+\frac{a}{4})_m}{(k+\frac{n}{2}-\frac{a}{4}+1)_m} \int_\RR \|T_{\mu,k}\|_{\HS(\calF_\mu')}^2 |\mu|^{n+s-2} \,d\mu\\
 &= \frac{2^{n}(-a)}{\pi^{n-1}} \sum_{k=0}^\infty \frac{(k+\frac{n}{2}-\frac{a}{4})(1+\frac{n}{2}-\frac{a}{4})_k}{(4k+2n-a)^2(\frac{n}{2}+\frac{a}{4})_k} \int_\RR \|T_{\mu,k}\|_{\HS(\calF_\mu')}^2 |\mu|^{n+s-2} \,d\mu\\
 &= \frac{2^{n-4}a}{\pi^{n-1}(\frac{a}{4}-\frac{n}{2})} \sum_{k=0}^\infty \frac{(\frac{n}{2}-\frac{a}{4})_k}{(\frac{n}{2}+\frac{a}{4})_k} \int_\RR \|T_{\mu,k}\|_{\HS(\calF_\mu')}^2 |\mu|^{n+s-2} \,d\mu\\
 &= \frac{\pi a}{a-2n}\|f\|_{-\frac{a}{2}}^2
\end{align*}
where we have used $(x)_{m+n}=(x)_m(x+m)_n$ in the third step and \eqref{eq:IdentitySumPochhammerQuotient} in the fourth step.

\subsection{$L^p$-$L^q$ boundedness}\label{sec:LpCplx}

We now show Theorem~\ref{thm:IsometryCplx}~(3).

\begin{proposition}
For any $1<p\leq\infty$ and $q=\frac{n+1}{n}p$ the Poisson transform $P_a$ is a bounded operator
$$ P_a: L^p(H^{2n-1})\to L^q(H^{2n+1}). $$
More precisely,
$$ \Vert P_af\Vert_q \leq (\pi\,c_{n,a}^\frac{1}{n})^{\frac{1}{q}}\Vert f\Vert_p, \qquad \mbox{for all $f\in L^p(H^{2n-1})$.} $$
\end{proposition}

\begin{proof}
By the Marcinkiewicz Interpolation Theorem it suffices to show that $P_a$ is a bounded operator
$$ L^\infty(H^{2n-1})\to L^\infty(H^{2n+1}) \qquad \mbox{and} \qquad L^1(H^{2n-1})\to L^{\frac{n+1}n}_w(H^{2n+1}), $$
where $L^r_w(H^{2n+1})$ stands for the weak type $L^r$-space. First observe that $P_a\1_{H^{2n-1}}=\1_{H^{2n+1}}$ by Lemma~\ref{lem:ExplicitBVConstantCplx} and that the integral kernel of $P_a$ is a positive function. Thus we have
$$ P_a: L^\infty(H^{2n-1})\to L^\infty(H^{2n+1}), \qquad \Vert P_af\Vert_{\infty}\le\Vert f\Vert_{\infty}. $$
We now prove the weak type inequality
$$ P_a: L^1(H^{2n-1})\to L^{\frac{n+1}n}_w(H^{2n+1}), \qquad \Vert P_af\Vert_{\frac{n+1}{n},w}\leq(\pi^nc_{n,a})^{\frac{1}{n+1}}\Vert f\Vert_1. $$
First note that
$$ |P_af(z,t)| \leq c_{n,a} \int_{H^{2n-1}} \frac{|z_n|^{-a}}{|(z,t)^{-1}\cdot(z',0,t')|^{2n-a}} |f(z',t')|\,d(z',t') \leq c_{n,a} |z_n|^{-2n} \Vert f\Vert_1. $$
Hence $|P_af(z,t)|>\lambda$ implies $|z_n|<(\frac{c_{n,a}\Vert f\Vert_1}{\lambda})^{\frac{1}{2n}}$. Denote by $m$ the Lebesgue measure on $H^{2n+1}$. Let $b=(\frac{c_{n,a}\Vert f\Vert_1}{\lambda})^{\frac{1}{2n}}$, then Chebyshev’s Inequality yields
\begin{align*}
 & m(\{(z,t)\in H^{2n+1}:|P_af(z,t)|>\lambda\}) = m(\{(z,t)\in H^{2n+1}:|z_n|<b,\,|P_af(z,t)|>\lambda\})\\
 \leq{}& \frac{1}{\lambda} \int_{(z,t)\in H^{2n+1},\,|z_n|<b} |P_af(z,t)|\,d(z,t)\\
 \leq{}& \frac{c_{n,a}}{\lambda} \int_{(z',t')\in H^{2n-1}} |f(z',t')| \int_{(z,t)\in H^{2n+1},\,|z_n|<b} \frac{|z_n|^{-a}}{|(z,t)^{-1}\cdot(z',0,t')|^{2n-a}} \,d(z,t)\,d(z',t')\\
 ={}& \frac{1}{\lambda}\int_{H^{2n-1}} |f(z',t')| \int_{|z_n|<b} P_a\1(z',z_n,t')\,dz_n\,d(z',t')\\
 ={}& \frac{1}{\lambda}\Vert f\Vert_1\int_{|z_n|<b}dz_n = \pi c_{n,a}^{\frac{1}{n}}\left(\frac{\Vert f\Vert_1}{\lambda}\right)^{\frac{n+1}{n}},
\end{align*}
where we have used $P_a\1=\1$ in the fifth step. Hence
$$ \Vert P_af\Vert_{\frac{n+1}{n},w} \leq (\pi^nc_{n,a})^{\frac{1}{n+1}}\Vert f\Vert_1 $$
and the claim follows.
\end{proof}

\appendix

\section{The full spectral decomposition in the real case}\label{sec:AppendixReal}

We describe the complete spectral decomposition of the operator $\Delta_a$ on $\dot{H}^{\frac{2-a}{2}}(\RR^n)$ for $2-n<a\leq2$.

\subsection{Inversion and Plancherel formula}

Recall the classical Gegenbauer polynomials $C_n^\alpha(z)$ given by
$$ C_n^\alpha(z) = \sum_{k=0}^{\lfloor\frac{n}{2}\rfloor} \frac{(-1)^k(\lambda)_{n-k}(2z)^{n-2k}}{k!(n-2k)!}. $$
We inflate the Gegenbauer polynomials to two-variable polynomials $C_n^\alpha(x,y)$ by setting
$$ C_n^\alpha(x,y) := x^{\frac{n}{2}}C_n^\alpha\left(\frac{y}{\sqrt{x}}\right). $$
For $2-n<a\leq2$ and $0\leq j<\frac{1-a}{2}$ we put
\begin{align*}
 D_{a,j}u(y) :={}& \frac{j!}{2^j(\frac{a-1}{2})_j}C_j^{\frac{a-1}{2}}\left(-\Delta_{x'},\frac{\partial}{\partial x_n}\right)u(y,0),\\
 P_{a,j}f(x) :={}& \frac{\Gamma(\frac{n-a}{2}-j)}{j!\pi^{\frac{n-1}{2}}\Gamma(\frac{1-a}{2}-j)} \int_{\RR^{n-1}} \frac{x_n^j|x_n|^{1-a-2j}}{(|x'-y|^2+x_n^2)^{\frac{n-a}{2}-j}} f(y) \,dy.
\end{align*}
Further, for $\nu\in\RR$ and $\varepsilon\in\ZZ/2\ZZ$ let
\begin{align*}
 D_{a,\nu,\varepsilon}u(y) :={}& \int_{\RR^n} \frac{\sgn(x_n)^\varepsilon |x_n|^{\frac{a-3}{2}+i\nu}}{(|x'-y|^2+x_n^2)^{\frac{n-1}{2}+i\nu}} u(x) \,dx,\\
 P_{a,\nu,\varepsilon}f(x) :={}& \int_{\RR^{n-1}} \frac{\sgn(x_n)^\varepsilon |x_n|^{\frac{1-a}{2}-i\nu}}{(|x'-y|^2+x_n^2)^{\frac{n-1}{2}-i\nu}} f(y) \,dy.
\end{align*}

It is easy to see that the operators $P_{a,j}$ and $P_{a,\nu,\varepsilon}$ produce eigenfunctions of $\Delta_a$:
\begin{align*}
 \Delta_a(P_{a,j}f) &= j(j+a-1)(P_{a,j}f),\\
 \Delta_a(P_{a,\nu,\varepsilon}f) &= -\left(\tfrac{1-a}{2}+i\nu\right)\left(\tfrac{1-a}{2}-i\nu\right)(P_{a,\nu,\varepsilon}f).
\end{align*}

\begin{theorem}\label{thm:FullDecompositionReal}
For $2-n<a\leq2$ we have
$$ u(x) = \sum_{j\in[0,\frac{1-a}{2})\cap\ZZ} P_{a,j}D_{a,j}u(x) + \frac{1}{4\pi^n}\sum_{\varepsilon=0,1} \int_0^\infty P_{a,\nu,\varepsilon}D_{a,\nu,\varepsilon}u(x) \left|\frac{\Gamma(\frac{n-1}{2}+i\nu)}{\Gamma(i\nu)}\right|^2\,d\nu $$
and
\begin{multline*}
 \|u\|_{\dot{H}^{\frac{2-a}{2}}(\RR^n)}^2 = \sum_{j\in[0,\frac{1-a}{2})\cap\ZZ} \frac{2^{a+2j}\pi\Gamma(2-a-j)}{j!(\frac{1-a}{2}-j)\Gamma(\frac{1-a}{2}-j)^2}\|D_{a,j}u\|_{\dot{H}^{\frac{1-a}{2}-j}(\RR^{n-1})}^2\\
 + \frac{1}{2^a\pi^n}\sum_{\varepsilon=0,1} \int_0^\infty \|D_{a,\nu,\varepsilon}u\|_{L^2(\RR^{n-1})}^2 \left|\frac{\Gamma(\frac{3-a+2\varepsilon+2i\nu}{4})\Gamma(\frac{n-1}{2}+i\nu)}{\Gamma(\frac{-1+a+2\varepsilon+2i\nu}{4})\Gamma(i\nu)}\right|^2\,d\nu.
\end{multline*}
\end{theorem}

The proof of this theorem is outlined in the next subsection.

\begin{corollary}\label{cor:MixedBVPReal}
Let $j\in\NN$ and $2-n<a<1-2j$.
\begin{enumerate}
\item For $f\in\dot{H}^{\frac{1-a}{2}-j}(\RR^{n-1})$ the mixed boundary value problem
\begin{equation}
 \Delta_au = j(j+a-1)u, \qquad D_{a,j}u = f\label{eq:MixedBVPReal}
\end{equation}
has a unique solution $u=P_{a,j}f\in\dot{H}^{\frac{2-a}{2}}(\RR^n)$ where $P_{a,j}:\dot{H}^{\frac{1-a}{2}-j}(\RR^{n-1})\to\dot{H}^{\frac{2-a}{2}}(\RR^n)$ is the integral operator
$$ P_{a,j}f(x) = \frac{\Gamma(\frac{n-a}{2}-j)}{j!\pi^{\frac{n-1}{2}}\Gamma(\frac{1-a}{2}-j)}\int_{\RR^{n-1}}\frac{x_n^j|x_n|^{1-a-2j}}{(|x'-y|^2+x_n^2)^{\frac{n-a}{2}-j}}f(y)\,dy. $$
\item The operator $P_{a,j}$ is isometric (up to a constant), more precisely
$$ \|P_{a,j}f\|_{\dot{H}^{\frac{2-a}{2}}(\RR^n)}^2 = \frac{2^{a+2j}\pi\Gamma(2-a-j)}{j!(\frac{1-a}{2}-j)\Gamma(\frac{1-a}{2}-j)^2}\|f\|_{\dot{H}^{\frac{1-a}{2}-j}(\RR^{n-1})}^2. $$
\end{enumerate}
\end{corollary}

\subsection{Reduction to an ordinary differential operator}

In Section~\ref{sec:UniquenessReal} we showed that the spectral decomposition of $\Delta_a$ on $\dot{H}^{\frac{2-a}{2}}(\RR^n)$ is via the Euclidean Fourier transform equivalent to the spectral decomposition of the self-adjoint ordinary differential operator
$$ \calD_{a,z}=(1+z^2)\frac{d^2}{dz^2}-(a-4)z\frac{d}{dz}-(a-2) \qquad \mbox{on} \qquad L^2(\RR,(1+z^2)^{\frac{2-a}{2}}dz). $$
This decomposition is calculated explicitly in \cite{MO12} using the Kodaira--Titchmarsh formula. We show how the results in \cite{MO12} translate to Theorem~\ref{thm:FullDecompositionReal}.

For $\sigma=a-2\in(-n,0)$ and $k=0,1$ let $\mu=1+2k$ and
$$ T(\sigma,k) := i\RR_+ \cup \{\sigma+\mu+4j:j\in[0,-\tfrac{\sigma+\mu}{4})\cap\ZZ\}. $$
For $\tau\in T(\sigma,k)$ and $u\in C_c^\infty(\RR^n)$ define
$$ F(\xi',\tau,k) = \hat{U}(\xi',\tau,k) := |\xi'|^{-\frac{\sigma-\tau+\mu}{2}} \int_\RR {_2F_1}\left(\frac{\sigma+\mu+\tau}{4},\frac{\sigma+\mu-\tau}{4};\frac{\mu}{2};-\frac{\xi_n^2}{|\xi'|^2}\right)\xi_n^kU(\xi',\xi_n) \,d\xi_n $$
and for $F\in C_c^\infty(\RR^m\times T(\sigma,k))$ let
$$ \check{F}(\xi,\tau,k) := \xi_n^k |\xi'|^{\frac{\sigma-\tau-\mu}{2}}{_2F_1}\left(\frac{\mu-\sigma+\tau}{4},\frac{\mu-\sigma-\tau}{4};\frac{\mu}{2};-\frac{\xi_n^2}{|\xi'|^2}\right)F(\xi',\tau,k). $$
Then by \cite[Theorem 4.1]{MO12} we have the decomposition of $U\in L^2(\RR^n,|\xi|^{-\sigma}d\xi)$ into eigenfunctions of $\calD_{\sigma+2,z}$:
$$ U(\xi) = \sum_{k=0,1} \int_{T(\sigma,k)} \check{F}(\xi,\tau,k) \,dm_{\sigma,k}(\tau)   $$
and the Plancherel formula
$$ \|U\|_{L^2(\RR^n,|\xi|^{-\sigma}dx)}^2 = \sum_{k=0,1} \int_{T(\sigma,k)} \|\hat{U}(\blank,\tau,k)\|_{L^2(\RR^{n-1},|\xi'|^{-\Re\tau}dx)}^2 \,dm_{\sigma,k}(\tau), $$
where
\begin{multline*}
 \int_{T(\sigma,k)} g(\tau) \,dm_{\sigma,k}(\tau) = \frac{1}{8\pi}\int_{i\RR_+} g(\tau)\left|\frac{\Gamma(\frac{\sigma+\mu+\tau}{4})\Gamma(\frac{-\sigma+\mu+\tau}{4})}{\Gamma(\frac{\tau}{2})\Gamma(\frac{\mu}{2})}\right|^2 \,d\tau\\
 +\sum_{j\in[0,-\frac{\sigma+\mu}{4})\cap\ZZ} \frac{(-1)^j\Gamma(-\frac{\sigma+2j}{2})\Gamma(\frac{\sigma+\mu}{2}+j)\Gamma(\frac{\mu}{2}+j)}{j!\Gamma(\frac{\mu}{2})^2\Gamma(-\frac{\sigma+\mu}{2}-2j)\Gamma(\frac{\sigma+\mu}{2}+2j)}g(\sigma+\mu+4j).
\end{multline*}

By the calculations in \cite[Section 5]{MO12} we obtain for $\tau\in i\RR_+$:
$$ \calF_{\RR^{n-1}}\hat{U}(y,\tau,k) = \frac{2^{\frac{\tau-\sigma}{2}}i^k\Gamma(\frac{\mu}{2})\Gamma(\frac{\tau+n-1}{2})}{\pi^{\frac{n-1}{2}}\Gamma(\frac{\sigma+\mu+\tau}{4})\Gamma(\frac{\sigma+\mu-\tau}{4})} \int_{\RR^n} \frac{x_n^k|x_n|^{\frac{\sigma+\tau-\mu}{2}}}{(|x-x'|^2+x_n^2)^{\frac{\tau+n-1}{2}}} \calF_{\RR^n}U(x) \,dx $$
and for $\tau=\sigma+\mu+4j$:
$$ \calF_{\RR^{n-1}}\hat{U}(y,\tau,k) = \frac{i^{-k}j!\sqrt{2\pi}}{(\frac{1-\sigma-2j}{2})_j(-1-\sigma-2j)_k} C_{2j+k}^{\frac{\sigma+1}{2}}\left(-\Delta_{x'},\frac{\partial}{\partial x_n}\right)\calF_{\RR^n}U(y,0) $$
and
$$ \calF_{\RR^n}\check{F}(x,\tau,k) = \frac{2^{\frac{\sigma-\tau}{2}}i^{-k}\Gamma(\frac{\mu}{2})\Gamma(\frac{n-1-\tau}{2})}{\pi^{\frac{n-1}{2}}\Gamma(\frac{\mu-\sigma+\tau}{4})\Gamma(\frac{\mu-\sigma-\tau}{4})} \int_{\RR^{n-1}} \frac{x_n^k|x_n|^{-\frac{\sigma+\tau+\mu}{2}}}{(|x'-y|^2+x_n^2)^{\frac{n-1-\tau}{2}}} \calF_{\RR^{n-1}}F(y,\tau,k) \,dy. $$

Then a short calculation yields that for $u\in\dot{H}^{-\frac{\sigma}{2}}(\RR^n)$:
\begin{multline*}
 u(x) = \sum_{j\in[0,-\frac{\sigma+1}{2})\cap\ZZ}(-1)^j
\frac{\Gamma(\frac{1-\sigma}{2}-j)\Gamma(\frac{n-\sigma-2}{2}-j)}{2^j\pi^{\frac{n-1}{2}}\Gamma(-\frac{\sigma+1}{2}-j)\Gamma(\frac{1-\sigma}{2})}
\\
\quad \times
\int_{\RR^{n-1}} \frac{x_n^j|x_n|^{-\sigma-1-2j}}{(|x'-y|^2+x_n^2)^{\frac{n-\sigma-2}{2}-j}} C_j^{\frac{\sigma+1}{2}}(-\Delta_{x'},\frac{\partial}{\partial x_n})u(y,0) \,dy\\
 +\sum_{\varepsilon=0,1}\frac{1}{8\pi^n}\int_{i\RR_+}\int_{\RR^{n-1}}\frac{\sgn(x_n)^\varepsilon|x_n|^{-\frac{\sigma+\tau+1}{2}}}{(|x'-y|^2+x_n^2)^{\frac{n-1-\tau}{2}}}\int_{\RR^n}\frac{\sgn(z_n)^\varepsilon|z_n|^{\frac{\sigma+\tau-1}{2}}}{(|z'-y|^2+z_n^2)^{\frac{n-1+\tau}{2}}}u(z)\,dz\,dy\left|\frac{\Gamma(\frac{n-1+\tau}{2})}{\Gamma(\frac{\tau}{2})}\right|^2\,d\tau
\end{multline*}
and
\begin{multline*}
 \|u\|_{\dot{H}^{-\frac{\sigma}{2}}(\RR^n)}^2 = \sum_{j\in[0,-\frac{\sigma+1}{2})\cap\ZZ} \frac{2^{\sigma+1}\pi j!(-\sigma-1-2j)\Gamma(-\sigma-j)}{\Gamma(\frac{1-\sigma}{2})^2}\|C_j^{\frac{\sigma+1}{2}}(-\Delta_{x'},\frac{\partial}{\partial x_n})u\|_{\dot{H}^{-\frac{\sigma+2j+1}{2}}(\RR^{n-1})}^2\\
 +\frac{1}{2^{\sigma+3}\pi^n}\sum_{\varepsilon=0,1} \int_{i\RR_+} \|D_{\sigma,\tau,\varepsilon}u\|_{L^2(\RR^{n-1}}^2\left|\frac{\Gamma(\frac{-\sigma+\tau+2\varepsilon+1}{4})\Gamma(\frac{n-1+\tau}{2})}{\Gamma(\frac{\sigma+\tau+2\varepsilon+1}{4})\Gamma(\frac{\tau}{2})}\right|^2\,d\tau
\end{multline*}
This yields Theorem~\ref{thm:FullDecompositionReal}.

\section{Part of the discrete spectrum in the complex case}\label{sec:AppendixCplx}

We construct explicitly eigenfunctions of $\calL_a$ in $\dot{H}^{\frac{2-a}{2}}(H^{2n+1})$ to the eigenvalues $2k(2k+a)$, $0\leq2k<-\frac{a}{2}$, and show that they are solutions to certain mixed boundary value problems.

\subsection{Mixed boundary value problems and their Poisson transforms}

Recall from \eqref{eq:CRLaplacian} the left-invariant CR-Laplacian on the Heisenberg group $H^{2n+1}$ and denote by $\calL'$ the left-invariant CR-Laplacian of the subgroup $H^{2n-1}\subseteq H^{2n+1}$ given by
$$ \calL' = \sum_{j=1}^{n-1} \left(\left(\frac{\partial}{\partial x_j}+2y_j\frac{\partial}{\partial t}\right)^2+\left(\frac{\partial}{\partial y_j}-2x_j\frac{\partial}{\partial t}\right)^2\right). $$
Further, write $\calT$ for the Laplacian on the center of $H^{2n+1}$, i.e.
$$ \calT = \frac{\partial^2}{\partial t^2}. $$

Following \cite{MOZ14} we inductively define a sequence $(\DD_{s,k})_k$ of differential operators on $H^{2n+1}$ depending on a parameter $s\in\CC$ by
$$ \DD_{s,0} := 1, \qquad \DD_{s,1} := \frac{1}{16s^2(2s+n)}\left[(2s+n-1)\calL-(2s+n)\calL'\right] $$
and
\begin{multline*}
 \DD_{s,k+1} := \frac{1}{16(s-k)^2(2s+n)}\Bigg[\Big((2s+n-2k-1)\calL-(2s+n)\calL'\Big)\DD_{s,k}\\
 -\frac{k^2(2s+n-2k-1)}{16s^2(2s+n-1)(2s+n)}\Big(\calL^2+16(2s+n)^2\calT\Big)\DD_{s-1,k-1}\Bigg].
\end{multline*}

Now for $-2n<a\leq2$ and $0\leq2k<-\frac{a}{2}$ we define a differential restriction operator $D_{a,k}:C^\infty(H^{2n+1})\to C^\infty(H^{2n-1})$ by
$$ D_{a,k}u(z',t') := (\DD_{-\frac{a+2n}{4},k}u)(z',0,t'). $$
In \cite[Theorem 4.1]{MOZ14} we show that $D_{a,k}$ extends to a bounded operator $\dot{H}^{\frac{2-a}{2}}(H^{2n+1})\to\dot{H}^{-\frac{a}{2}-2k}(H^{2n-1})$.

\begin{theorem}\label{thm:MixedBVPCplx}
Let $k\in\NN$ and $-2n<a<-4k$.
\begin{enumerate}
\item For $f\in\dot{H}^{-\frac{a}{2}-2k}(H^{2n-1})$ the mixed boundary value problem
\begin{equation}
 \calL_au = 2k(2k+a)u, \qquad D_{a,k}u = f\label{eq:MixedBVPCplx}
\end{equation}
has a solution $u\in\dot{H}^{\frac{2-a}{2}}(H^{2n+1})$. More precisely, there exists a constant $c_{n,a,k}$ such that the operator $P_{a,k}:\dot{H}^{-\frac{a}{2}-2k}(H^{2n-1})\to\dot{H}^{\frac{2-a}{2}}(H^{2n+1})$ given by
$$ P_{a,k}f(z,t) = c_{n,a,k}\int_{H^{2n-1}}\frac{|z_n|^{-a-2k}}{|(z,t)^{-1}\cdot(z',0,t')|^{2n-a-4k}}f(z',t')\,d(z',t') $$
constructs a solution $u=P_{a,k}f$ of \eqref{eq:MixedBVPCplx}.
\item The operator $P_{a,k}$ is isometric (up to a constant), i.e. there exists a constant $C>0$ such that
$$ \|P_{a,k}f\|_{\dot{H}^{\frac{2-a}{2}}(H^{2n+1})}^2 = C\|f\|_{\dot{H}^{-\frac{a}{2}-2k}(H^{2n-1})}^2. $$
\end{enumerate}
\end{theorem}

Clearly part (1) of the theorem implies that $2k(2k+a)$ is an eigenvalue of $\calL_a$ on $\dot{H}^{\frac{2-a}{2}}(H^{2n+1})$ as claimed in Theorem~\ref{thm:DirichletCplx}~(1).

We prove this theorem in the next subsection.

\subsection{Discrete components in restrictions of complementary series}

Let $\mu=\frac{a-2}{2}$ and $\nu=\frac{a}{2}+2k$ so that $I_\mu=\dot{H}^{\frac{2-a}{2}}(H^{2n+1})$ and $J_\nu=\dot{H}^{-\frac{a}{2}-2k}(H^{2n-1})$. In \cite[Theorem 5.2]{MOZ14} we show that the operators $D_{a,k}:I_\mu\to J_\nu$ are intertwining the representations $\pi_\mu|_{G'}$ and $\tau_\nu$ of $G'$, taking the role of the trace map in Section~\ref{sec:TraceMaps}. Just as in the case of the trace map one shows the following statements:
\begin{enumerate}
\item The adjoint operator $D_{a,k}^*:J_\nu\to I_\mu$ embeds $(\tau_\nu,J_\nu)$ isometrically as a subrepresentation of $(\pi_\mu|_{G'},I_\mu)$ and is hence up to a constant equal to the symmetry breaking operator $B_{\mu,\nu}$.
\item The operator $\calL_a$ acts on the image of $B_{\mu,\nu}$ by the scalar
$$ -((\mu+\rho)-(\nu+\rho'))((\mu+\rho)+(\nu-\rho')) = 2k(2k+a). $$
\item The composition $D_{a,k}\circ B_{\mu,\nu}:J_\nu\to J_\nu$ is a scalar multiple of the identity.
\end{enumerate}

Now we are ready to prove Theorem~\ref{thm:MixedBVPCplx}.

\begin{proof}[{Proof of Theorem~\ref{thm:MixedBVPCplx}}]
By (1) and (3) there exists a constant $c_{n,a,k}$ such that the operator $P_{a,k}:=c_{n,a,k}\cdot B_{\mu,\nu}$ has the property that $D_{a,k}\circ P_{a,k}=\id$. Further, by (2) the image of $P_{a,k}$ consists of eigenfunctions of $\calL_a$ to the eigenvalue $2k(2k+a)$, and hence $P_{a,k}$ constructs solutions to the mixed boundary value problem \eqref{eq:MixedBVPCplx}. This shows Theorem~\ref{thm:MixedBVPCplx}~(1). Isometry of $P_{a,k}$ (up to a constant) then follows from (1).
\end{proof}

\bibliographystyle{amsplain}
\bibliography{bibdb}

\end{document}